\documentclass[10pt]{amsart}

\usepackage{amssymb}
\usepackage{bbm}
\usepackage{enumerate, xspace}

\let\mathds\mathbf
\usepackage{hyperref}

\newcommand{\one}{{\mathds{1}}}

\let\epsilon\varepsilon

\newcommand\phistar{\phi_{x_0}}  

\newcommand\ds{\displaystyle}

\newcommand\var{{\rm Var}}

\newcommand\dto{\, {\stackrel{d}{\to} \,}}

\numberwithin{equation}{section}

\newtheorem{thm}{Theorem}[section]
\newtheorem{lem}[thm]{Lemma}
\newtheorem{cor}[thm]{Corollary}

\newtheorem{prop}[thm]{Proposition}

\newtheorem{defin}[thm]{Definition}

\newtheorem{exmp}[thm]{Example}
\newtheorem{rem}[thm]{Remark}

\newcommand\cA{{\mathcal A}}
\newcommand\cB{{\mathcal B}}

\newcommand\cD{{\mathcal D}}
\newcommand\cE{{\mathcal E}}
\newcommand\cF{{\mathcal F}}

\newcommand\cJ{{\mathcal J}}

\newcommand\cO{{\mathcal O}}

\newcommand\cR{{\mathcal R}}
\newcommand\cS{{\mathcal S}}
\newcommand\cT{{\mathcal T}}

\newcommand\bD{{\mathbb D}}
\newcommand\bE{{\mathbb E}}

\newcommand\bP{{\mathbb P}}

\newcommand\bR{{\mathbb R}}

\newcommand\bZ{{\mathbb Z}}
\newcommand{\Z}{\mathbb{Z}}
\newcommand{\R}{\mathbb{R}}

\begin{document}

\date{\today.}

\title{Stable laws for random dynamical systems}

\author[R. Aimino]{Romain Aimino}
\address{Romain Aimino\\
Departamento de Matem\'atica\\
Faculdade de Ci\^encias da Universidade do Porto\\
Rua do Campo Alegre, 687, 4169-007 Porto, Portugal.}
\email{{\tt romain.aimino@fc.up.pt}}
\urladdr{https://www.fc.up.pt/pessoas/romain.aimino/}

\author[M. Nicol]{Matthew Nicol}
\address{Matthew Nicol\\ Department of Mathematics\\
University of Houston\\
Houston\\
TX 77204\\
USA} \email{nicol@math.uh.edu}
\urladdr{http://www.math.uh.edu/~nicol/}

\author[A. T\"or\"ok]{Andrew T\"or\"ok}
\address{Andrew T\"or\"ok\\ Department of Mathematics\\
  University of Houston\\
  Houston\\
  TX 77204\\
  USA and 
{Institute of Mathematics of the Romanian Academy, Bucharest, Romania.}}
\email{torok@math.uh.edu}
\urladdr{http://www.math.uh.edu/~torok/}

\thanks{RA was partially supported by FCT project PTDC/MAT-PUR/4048/2021,
  with national funds, and by CMUP, which is financed by national funds
  through FCT -- Funda\c{c}\~ao para a Ci\^encia e a Tecnologia, I.P.,
  under the project with reference UIDB/00144/2020. MN was supported in
  part by NSF Grants DMS 1600780 and DMS 2009923. AT was supported in part
  by NSF Grant DMS 1816315. RA would like to thank Jorge Freitas for
  several very insightful discussions about return time statistics and
  point processes. We wish to thank an anonymous referee for helpful comments.
  Data sharing not applicable to this article as no datasets were generated or analysed during the current study.}



\begin{abstract}

  In this paper we consider random dynamical systems formed by
  concatenating maps acting on the unit interval $[0,1]$ in an iid fashion.
  Considered as a stationary Markov process, the random dynamical system
  possesses a unique stationary measure $\nu$. We consider a class of
  non square-integrable observables $\phi$, mostly of form
  $\phi(x)=d(x,x_0)^{-\frac{1}{\alpha}}$ where $x_0$ is a non-recurrent point (in particular a non-periodic point)
  satisfying some other genericity conditions, and more generally regularly
  varying observables with index $\alpha \in (0,2)$. The two types of maps we concatenate are a
  class of piecewise $C^2$ expanding maps, and a class of intermittent maps
  possessing an indifferent fixed point at the origin. Under conditions on
  the dynamics and $\alpha$ we establish Poisson limit laws, convergence of
  scaled Birkhoff sums to a stable limit law and functional stable limit
  laws, in both the annealed and quenched case. The scaling constants for
  the limit laws for almost every quenched realization are the same as
  those of the annealed case and determined by $\nu$. This is in contrast
  to the scalings in quenched central limit theorems where the centering
  constants depend in a critical way upon the realization and are not the
  same for almost every realization.

\end{abstract}

\maketitle

\tableofcontents


%

\maketitle

\section{Introduction}\label{sec:intro}

In this paper we consider non square-integrable observables $\phi: [0,1]\to \bR$
on two simple classes of random dynamical system. One consists of randomly
choosing in an iid manner from a finite set of maps which are strictly
polynomially mixing with an indifferent fixed point at the origin, the
other consisting of randomly choosing from a finite set of maps which are
uniformly expanding and exponentially mixing. The main type of observable
we consider is of the form $\phi (x) = |x-x_0|^{-\frac{1}{\alpha}}$,
$\alpha\in (0,2)$ which in the IID case lies in the domain of attraction of
a stable law of index $\alpha$. For certain results the point $x_0$ has to
satisfy some nongenericity conditions and in particular not be a
periodic point for almost every realization of the random system (see
Definition~\ref{def:preperiodic}). Some of our results, particularly those involving convergence to
exponential and Poisson laws hold for general observables
that are regularly varying with index $\alpha$.

The settings for investigations on   stable limit laws for observables on dynamical systems tend to be of two broad types:  (1) ``good observables'' (typically H\"older) on 
slowly mixing non-uniformly hyperbolic systems;  and  (2) ``bad'' observables (unbounded with fat tails) on fast mixing dynamical systems. 
As illustrative examples of both settings we give two results.

\noindent {\bf Example of (1):} 
The LSV intermittent map $T_{\gamma}: [0,1]\to [0,1]$, $\gamma \in (0,1)$,
is defined by
\[
  T_{\gamma} (x)  = \left\{ \begin{array}{ll}
         x(1+2^{\gamma}x^{\gamma}) & \mbox{if $0 \leq x\le \frac{1}{2}$};\\
        2x-1 & \mbox{if $\frac{1}{2}  <x < 1$}.\end{array} \right. 
\] 
        
The map $T_{\gamma}$ has a unique absolutely continuous invariant measure
$\mu_{\gamma}$.
        
Gou\"ezel~\cite[Theorem 1.3]{Gouezel_Intermittent} showed that if $\gamma >\frac{1}{2}$ and
$\phi: [0,1]\to \bR$ is H\"older continuous with $\phi (0)\not =0$,
$E_{\mu_{\gamma}}(\phi)=0$ then for $\alpha=\frac{1}{\gamma}$
\[
  \frac{1}{bn^{\frac{1}{\alpha}} }\sum_{j=0}^{n-1} \phi \circ T^j
  \to^{d} X_{\alpha,\beta}
\] 
($\beta$ has a complicated expression).

 \noindent {\bf Example of (2):} Gou\"ezel~\cite{G1}[Theorem 2.1] showed
that if $T:[0,1] \to [0,1]$ is the doubling map $T(x)=2x$ (mod $1$) with
invariant measure $m$, Lebesgue, and $\phi (x) =x^{-{\frac{1}{\alpha}}}$,
$\alpha \in (0,2)$ then there exists a sequence $c_n$ such that
\[
  \frac{2^{\frac{1}{\alpha}}-1}{n^{\frac{1}{\alpha}}} \sum_{j=0}^{n-1} \phi
  \circ T^j -c_n \to^{d} X_{\alpha,1}
\]

For further results on the 
first type we refer to the influential papers~\cite{Gouezel_Intermittent,Gouezel_Skew} and~\cite{Melbourne_Zweimuller}. In the setting of ``good observables'' (typically H\"older) on  slowly mixing non-uniformly hyperbolic systems the technique of inducing on a subset of phase space and 
constructing a Young Tower has been used with some  success. ``Good'' observables lift to well-behaved observables lying in a suitable 
Banach space on the Young Tower. This is not the case with unbounded observables with fat tails, though in~\cite{Gouezel_Intermittent} the induction technique
allows an observable to be unbounded at the fixed point in a family of intermittent maps.

For further results on the second type we refer to the papers by Marta
Tyran-Kaminska~\cite{TK,TK-dynamical}. In the setting of Gibbs-Markov maps
she shows, among other results, that functions which are measurable with
respect to the Gibbs-Markov partition and in the domain of attraction of a
stable law with index $\alpha$ converge (under the appropriate scaling) in
the $J_1$ topology to a L\'evy process of index $\alpha$~\cite[Theorem 3.3,
Corollaries 4.1 and 4.2]{TK-dynamical}.

 For recent results on limit laws, though not stable laws,
in the setting of skew-products with an ergodic base map and uniformly hyperbolic fiber maps see also~\cite{Froyland1, Froyland2}. For a still very useful survey of techniques and ideas in random dynamical systems we refer to~\cite{Kifer_1998}.

Our main results are given in Section 2.  An introduction to stable laws and a discussion of modes
of convergence is given in Sections 3 and 4. The Poisson point approach and its application to 
our random setting is detailed in Section 5. Results on convergence of return times to an exponential law
and our point processes to a Poisson process are given in Section 6 (though the proofs of these
results are delayed until sections 8.1, 8.2, 9.1 and 9.2).  The proofs of the main results are given in Section 10. We conclude in Section 11 with 
results on stable laws for the corresponding  annealed systems.

\section{Main Results}\label{Main_Results}

For the sake of concreteness, we restrict  ourselves to observables of
  the form 
\begin{equation}\label{eq:phi-star}
  \phistar(x) = |x - x_0|^{ - \frac{1}{\alpha}}, \: x \in [0,1].
\end{equation}
where $x_0$ is a non-recurrent point (see Definition~\ref{def:preperiodic}) and $\alpha \in (0,2)$
 but it is possible
  to consider more general regularly varying observables $\phi$ which are piecewise monotonic
  with finitely many branches, see for instance \cite[Section
  4.2]{TK-dynamical} in the deterministic case. 
 Note that $\phistar$ is regularity varying with index $\alpha$.

 We will be considering the following set-up, with $(\Omega, \sigma)$ the
full two-sided shift on finitely many symbols. In most of our settings we take $Y=[0,1]$.

Let $\sigma : \Omega \to \Omega$ be an invertible ergodic
measure-preserving transformation on a probability space
$(\Omega, \cF, \bP)$. For a measurable space $(Y, \cB)$, let $\sigma:\Omega \to \Omega$
be the usual full shift and define 
\[
F : \Omega \times Y  \to  \Omega \times Y 
 \]
by
\[
F(\omega, x) =(\sigma \omega, T_\omega (x)) 
\]
We assume $F$ preserves a probability measure $\nu$ on $\Omega \times Y$. We assume
that $\nu$ admits a disintegration given by
$\nu (d \omega, dx) = \bP(d \omega) \nu^\omega(dx)$. For all $n \ge 1$,
we have
\[
F^n(\omega, x) = (\sigma^n \omega, T_\omega^n x),
\]
where
\begin{equation*} 
  T_\omega^n = T_{\sigma^{n-1} \omega} \circ \ldots \circ T_\omega,
\end{equation*}
which satisfies the equivariance relations
$(T_\omega^n)_{*} \nu^\omega = \nu^{\sigma^n \omega}$ for $\bP$-a.e.
$\omega \in \Omega$.

For each $\omega \in \Omega$, we denote by $P_\omega$ the transfer operator of
$T_\omega$ with respect to the Lebesgue measure $m$: for all
$\phi \in L^\infty(m)$ and $\psi \in L^1(m)$,
\[
  \int_{[0,1]} (\phi \circ T_\omega) \cdot \psi \, dm = \int_{[0,1]} \phi
  \cdot P_\omega \psi \, dm.
\]
We can then form, for $\omega \in \Omega$ and $n \ge 1$, the cocycle
\[
P_\omega^n = P_{\sigma^{n-1} \omega} \circ \ldots \circ P_\omega.
\]

\begin{defin}[scaling constants]\label{def:scaling-constants}
  We consider a sequence $(b_n)_{n \ge 1}$ of positive real numbers such
  that
  \begin{equation}\label{eqn:tail1}
    \lim_{n \to \infty} n \nu( \phistar  > b_n ) = 1.
  \end{equation}
\end{defin}

\begin{defin}[centering constants]\label{def:centering-constants}
  
  We define the centering sequence $(c_n)_{n \ge 1}$ by
    \[
    c_n =
    \begin{cases}
      0 &\text{if $\alpha \in (0,1)$}\cr n \mathbb{E}_\nu(\phistar
      \one_{\left\{ \phistar  \le b_n\right\}}) & \text{if
        $\alpha = 1$} \cr n \mathbb{E }_\nu(\phistar) & \text{if
        $\alpha \in(1,2)$}
    \end{cases}.
  \]
\end{defin}

 We now introduce two classes of random dynamical system (RDS) for which
 we are able to establish stable limit laws.  
 
 \subsection{Random uniformly expanding maps}\label{sec:expanding_iid}
 
 We consider random i.i.d. compositions 
with additional assumptions of uniform expansion. Let $\cS$ be a finite
collection of $m$ piecewise $C^2$ uniformly expanding maps of the unit
interval $[0,1]$. More precisely, we assume that for each $T \in \cS$,
there exist a finite partition $\cA_T$ of $[0,1]$ into intervals, such that
for each $I \in \cA_T$, $T$ can be continuously extended as a strictly
monotonic $C^2$ function on $\bar{I}$ and
\[
\lambda := \inf_{I \in \cA_T} \inf_{x \in \bar{I}} |T'(x)| > 1.
\]

The maps $T_\omega$ (determined by the $0$-th
coordinate of $\omega$) are chosen from $\cS$ in an i.i.d. fashion
according to a Bernoulli probability measure $\bP$ on
$\Omega := \left\{1, \ldots, m\right\}^{\bZ}$. We will denote by
$\cA_\omega$ the partition of monotonicity of $T_\omega$, and by
$\cA_\omega^n = \vee_{k=0}^{n-1} (T_\omega^k)^{-1}(\cA_{\sigma^k \omega})$
the partition associated to $T_\omega^n$. We introduce
\[
  \cD = \cup_{n \ge 0} \cup_{\omega \in \Omega} \partial \cA_\omega^n
\]
the set of discontinuities of all the maps $T_\omega^n$. Note that $\cD$ is at most a countable set.

 In the uniformly expanding case we also assume the conditions 
 (LY),(Dec) and (Min). (LY) is the usual Lasota-Yorke inequality while (Dec) and (Min) were introduced by Conze and Raugi \cite{CR07}.

\begin{description}
\item[(LY)]
  there exist $r \ge 1$, $M>0$ and $D > 0$ and $\rho \in (0,1)$
  such that for all $\omega \in \Omega$ and all $f \in {\rm BV}$,
  \[
    \|P_\omega f\|_{\rm BV} \le M \|f\|_{\rm BV},
  \]
  and
  \[
    \var(P_\omega^r f) \le \rho \var(f) + D \|f\|_{L^1(m)}.
  \]
\end{description}



\begin{description}
\item[(Dec)] there exists $C >0$ and $\theta \in (0,1)$ such that for all $n \ge 1$, all $\omega \in  \Omega$ and all $f \in {\rm BV}$ with $\bE_m(f) = 0$:
\[
\| P_\omega^n f \|_{\rm BV} \le C \theta^n \|f\|_{\rm BV}
\]
\end{description}
\begin{description}
\item[(Min)] there exists $c>0$ such that for all $n\ge 1$ and all $\omega \in \Omega$,
\[
  \inf_{x \in [0,1]} (P_\omega^n \one)(x) \ge c>0.
\] 
\end{description}

\begin{defin}\label{def:preperiodic}
  We say that $x_0$ is non-recurrent if $x_0$  satisfies the condition $T^n_{\omega} (x_0)\not =x_0$ for all $n\ge1$ for $\bP$-a.e.  $\omega \in \Omega$. 
  \end{defin}

\begin{thm}\label{thm:expanding}
    In the setting of exapnding maps assume  (LY),  (Min) and (Dec).
    Suppose that $x_0 \notin \cD$ is non-recurrent and consider the observable $\phistar$.

    If $\alpha \in (0,1)$ then for $\bP$-a.e.~$\omega \in \Omega$, the
    Functional Stable Limit holds:
    \[
      X_n^\omega(t) := \frac{1}{b_n} [\sum_{j=0}^{\lfloor nt \rfloor - 1 } \phistar \circ
      T_\omega^j - t c_n] \: \dto X_{(\alpha)}(t) \text{ \quad in \quad }
      \bD[0, \infty)
    \]
      in the $J_1$ topology 
    under the probability measure $\nu^\omega$, where $X_{(\alpha)}(t)$ is
    the $\alpha$-stable process with L\'evy measure
    $d\Pi_\alpha(dx) = \alpha |x|^{-(\alpha + 1)}$ on $[0,\infty)$.

    If $\alpha \in [1,2)$ then the same result holds for 
    $m$-a.e.~$x_0$.
\end{thm}

\begin{exmp}[$\beta$-transformations]\label{beta}
  A simple example of a class of maps satisfying (LY), (Dec) and (Min)~\cite{CR07}  is to take $m$ $\beta$-maps of the unit interval,
  $T_{\beta_i } (x)=\beta_i x$ (mod $1$). We suppose $\beta_i >1+a$, $a>0$,
  for all $\beta_i$, $i=1,\ldots,m$.
  \end{exmp}

\subsection{Random intermittent maps}\label{intermittent}

Now we consider a simple class of intermittent type maps.

  Liverani, Saussol and Vaienti \cite{LivSauVai} introduced the map $T_{\gamma}$ as a
  simple model for intermittent dynamics:
  \[
    \text{$T_{\gamma}:[0,1] \to [0,1]$, \qquad } T_{\gamma} (x) := \left\{
      \begin{array}{ll}
         (2^{\gamma}x^{\gamma} +1)x& \mbox{if $0 \leq x<\frac{1}{2}$};\\
        2x-1 & \mbox{if $\frac{1}{2} \le x \le 1$}.
      \end{array} \right.
  \]
  If $0\le \gamma <1$ then $T_{\gamma}$ has an absolutely continuous
  invariant measure $\mu_{\gamma}$ with density $h_{\gamma}$ bounded away
  from zero and satisfying $h_{\gamma} (x)\sim C x^{-\gamma}$ for $x$ near
  zero.
 
  We form a random dynamical system by selecting $\gamma_i\in (0,1)$,
  $i=1,\ldots,m$ in an iid fashion and setting $T_i:=T_{\gamma_i}$.
  The associated Markov process on $[0,1]$ has a stationary invariant
  measure $\nu$ which is absolutely continuous, with density $h$ bounded
  away from zero.

  We denote $\gamma_{max}:=\max_{1\le i \le m} \{\gamma_i\}$ and
  $\gamma_{min}:=\min_{1\le i \le m} \{\gamma_i\}$.


\begin{thm}\label{thm:intermittent}
  In the setting of iid random composition of intermittent maps suppose $\alpha \in (0,1)$
  and $\gamma_{max}<\frac{1}{3}$. Then, for $m$-a.e.~$x_0$
  $\frac{1}{b_n}\sum_{j=0}^{n-1}\phistar\circ T_{\omega}^j \dto
  X_{(\alpha)}(1)$ under the probability measure $\nu^{\omega}$ for
  $\bP$-a.e.~$\omega$ (recall that $c_n=0$ for $\alpha \in (0,1)$).
\end{thm}

\begin{rem}[Convergence with respect to Lebesgue measure] We state our
  limiting theorems with respect to the fiberwise measures $\nu^{\omega}$
  but by general results of Eagleson~\cite{Eagleson}(see
  also~\cite{Zweimuller}) the convergence holds with respect to any measure
  $\mu$ for which $\mu \ll \nu^{\omega}$, in particular our convergence
  results hold with respect to Lebesgue measure $m$. Further details are
  given in the Appendix.
\end{rem}

 Our proofs are based on a Poisson process approach developed for dynamical systems by Marta Tyran-Kaminska~\cite{TK, TK-dynamical}.

\section{Probabilistic tools}\label{sec:tools}

In this section, we review some topics from Probability Theory.

\subsection{Regularly varying functions and domains of
  attraction}\label{ssec:regu}

We refer to
Feller~\cite{Fel71} or Bingham, Goldie and
Teugels~\cite{Bingham-Goldie-Teugels-1987}  for the relations between domains of attraction of
stable laws and regularly varying functions.  For $\phi$ regularly varying  we define the constants $b_n$ and 
$c_n$ as in the case of $\phi_{x_0}$.

\begin{rem}\label{rem:centering-constants}
  When $\alpha \in (0,1)$ then $\phi$ is not integrable and one can choose
  the centering sequence $(c_n)$ to be identically $0$. When $\alpha = 1$,
  it might happen that $\phi$ is not integrable, and it is then necessary
  to define $c_n$ with suitably truncated moments as above. If
  $\phi$ is integrable then center by  $c_n = n\bE_\nu(\phi)$.
\end{rem}

We will use the following asymptotics for truncated moments, which can be
deduced from Karamata's results concerning the tail behavior of regularly
varying functions.  Define $p$ by $ \lim_{x \to \infty}
    \frac{\nu(\phi>x)}{\nu(|\phi|>x)} = p$.

\begin{prop}[Karamata]\label{prop:karamata}

  Let $\phi$ be regularly varying with index $\alpha \in (0,2)$. Then,
  setting $\beta := 2p -1$  and, for $\epsilon > 0$,
  \begin{equation}\label{eq:c_alpha}
    c_\alpha(\epsilon) :=
    \begin{cases}
      0 &\text{if $\alpha \in (0,1)$}\cr
      - \beta \log \epsilon & \text{if $\alpha = 1$} \cr
      \epsilon^{1 - \alpha} \beta \alpha / (\alpha - 1) & \text{if
        $\alpha \in(1,2)$}
    \end{cases}
  \end{equation}
  the following hold for all $\epsilon>0$:
\begin{enumerate}[(a)]
\item $\displaystyle \bE_\nu(| \phi|^2 \one_{\left\{ | \phi | \le \epsilon b_n \right\}}) \sim \frac{\alpha}{2 - \alpha} (\epsilon b_n)^2 \nu(| \phi | > \epsilon b_n),$
\item if $\alpha \in (0,1)$, $$\bE_\nu(| \phi| \one_{\left\{ | \phi | \le \epsilon b_n \right\}}) \sim \frac{\alpha}{1 - \alpha} \epsilon b_n \nu(| \phi | > \epsilon b_n),$$
\item if $\alpha \in (1,2)$, $$\lim_{n \to \infty} \frac{n}{b_n} \bE_\nu(\phi \one_{\{ | \phi | > \epsilon b_n \}}) = c_\alpha(\epsilon),$$
\item if $\alpha = 1$, $$\lim_{n \to \infty} \frac{n}{b_n} \bE_\nu(\phi \one_{\{  \epsilon b_n < |\phi | \le  b_n \}}) = c_\alpha(\epsilon),$$
\item if $\alpha = 1$, $$\frac{n}{b_n} \bE_\nu(|\phi| \one_{ \{  | \phi | \le \epsilon b_n \}}) \sim \widetilde{L}(n),$$ for a slowly varying function $\widetilde{L}$,
\end{enumerate}
\end{prop}


\subsection{L\'evy $\alpha$-stable processes}\label{ssec:stable}

A helpful and more detailed discussion can be found, e.g., in~\cite{TK, TK-dynamical}.

$X(t)$ is a L\'evy stable process if $X(0)=0$, $X$ has stationary independent increments and $X(1)$ has an 
$\alpha$-stable distribution.  

The L\'evy-Khintchine representation for the characteristic function of an
$\alpha$-stable random variable $X_{\alpha,\beta}$ with index
$\alpha \in (0,2)$ and parameter $\beta\in[-1,1]$ has the form:
\[
\bE[e^{itX}]=\mbox{exp}\left[ita_{\alpha}+\int (e^{itx}-1-itx1_{[-1,1]} (x))\Pi_{\alpha}(dx)\right]
\]
where 
\begin{itemize}
\item 

  $\ds
  a_{\alpha}= \left\{ \begin{array}{ll}   \beta\frac{\alpha}{1-\alpha} & \mbox{ $\alpha \not = 1$}\\
                        0 & \mbox{ $\alpha =1$}\end{array} \right.,$

\item  $\Pi_{\alpha}$ is a L\'evy measure given by
  \[
    d \Pi_{\alpha} =\alpha (p1_{(0,\infty)} (x)+(1-p)1_{(-\infty,0)} (x) )
    |x|^{-\alpha -1} dx
 \]

\item $ \ds p=\frac{\beta+1}{2}.$
\end{itemize}

Note that $p$ and $\beta$ may equally serve as parameters for $X_{\alpha,\beta}$. We will drop the $\beta$ from $X_{\alpha,\beta}$, as is common in the literature, for simplicity of notation and when it plays no essential role.

\subsection{Poisson point processes}\label{ssec:poisson_process}

Let $(T_n)_{n \ge 1}$ be a sequence of measurable transformations on a
probability space $(Y, \cB, \mu)$. For $n \ge 1$ we denote
\begin{equation}\label{eq.sequential-composition}
  T_1^n := T_n \circ \ldots \circ T_1.
\end{equation}
Given $\phi : Y \to \bR$ measurable, recall that we define the scaled
Birkhoff sum by
\begin{equation}\label{eq.sequential}
  S_n:= \frac{1}{b_n} [\sum_{j=0}^{ n - 1} \phi \circ
  T_1^j -  c_n],
\end{equation}
for some real constants $b_n > 0$, $c_n$
and the scaled random process $X_n(t)$, $n\ge 1$, by
\begin{equation}\label{eq.sequential}
  X_n(t) := \frac{1}{b_n} [\sum_{j=0}^{\lfloor nt \rfloor - 1} \phi \circ
  T_1^j - t c_n], \; t \ge 0,
\end{equation}

For $X_{\alpha}(t)$ a L\'evy $\alpha$-stable process and
$B\in \mathcal{B} ((0,\infty) \times (\bR \setminus \{0\}))$ define
\[
  N_{(\alpha)}(B) := \#\{s>0: (s, \Delta X_{\alpha}(s))\in B\}
\]
where $\Delta X_{\alpha}(t):=X_{\alpha}(t)-X_{\alpha}(t^-)$.

The random variable $N_{(\alpha)}(B)$, which counts the jumps (and their
time) of the L\'evy process that lie in $B$, is finite a.s. if and only if
$(m \times \Pi_\alpha) (B)< \infty$. In that case $N_{(\alpha)}(B)$ has a
Poisson distribution with mean $(m \times \Pi_\alpha) (B)$.

Similarly define
\[
  N_n(B) := \# \left\{ j \ge 1 \, : \, \left(\frac{j}{n} , \frac{\phi
\circ T_1^{j-1}}{b_n} \right) \in B \right\}, \; n \ge 1,
\]
$N_n(B)$ counts the jumps of the process~\eqref{eq.sequential} that lie in
$B$. When a realization $\omega \in \Omega $ is fixed we define
\[
N_n^\omega(B) := \# \left\{ j \ge 1 \, : \, \left(\frac{j}{n} , \frac{\phi
      \circ T_\omega^{j-1}}{b_n} \right) \in B \right\}, \; n \ge 1.
\]

\begin{defin}
  We say $N_n$ converges in distribution to  $N_{(\alpha)}$ and write
  \[
    N_n \dto N_{(\alpha)}
  \]
  if and only if $N_n(B) \dto N_{(\alpha)}(B)$ for all
  $B \in B((0, \infty) \times (\bR \setminus \{ 0 \}))$ with
  $(m \times \Pi_\alpha) (B) < \infty$ and
  $(m \times \Pi_\alpha ) (\partial B) = 0$.
\end{defin}

\section{Modes of Convergence}

Consider the process $X_{\alpha}$ determined by the observable
$\phi$ (that is, an iid version of $\phi$ which regularly varying with the same index $\alpha$ and
parameter $p$ ).
%
%
We are interested the following limits:

\begin{itemize}
\item [(A)] \textbf{Poisson point process convergence.}
  \[
    ~N^{\omega}_n \dto N_{(\alpha)}
  \]
  with respect to $\nu^{\omega}$ for $\bP$ a.e.~$\omega$ where $N_{(\alpha)}$
  is the Poisson point process of an $\alpha$-stable process with parameter
  determined by $\nu$, the annealed measure.


\item [(B)] \textbf{Stable law convergence.}
  \[
    ~S_n^{\omega}:=\frac{1}{b_n} [\sum_{j=0}^{n-1} \phi\circ T_{\omega}^j
    -c_n] \dto X_{\alpha}(1)
  \]
  for $\bP$-a.e.~$\omega$, with respect to $\nu^{\omega}$, for $\phi$
  regularly varying with index $\alpha$ and $X_{\alpha}(t)$ the
  corresponding $\alpha$-stable process,
  for suitable scaling and centering constants $b_n$ and $c_n$.
  
   \item [(C)] \textbf{Functional stable law convergence.}
  \[
    ~X_n^{\omega} (t) := \frac{1}{b_n}[ \sum_{j=0}^{\lfloor nt \rfloor-1} \phi \circ
    T_{\omega}^j - tc_n] \dto X_{\alpha} (t)
  \]
  in $\bD[0,\infty)$ in the $J_1$ topology $\bP$-a.e.~$\omega$, with respect to $\nu^{\omega}$ for $\phi$
  regularly varying with index $\alpha$ and $X_{\alpha}(t)$ the
  corresponding $\alpha$-stable process.

\end{itemize}

For the cases we are considering, the scaling constants $b_n$ are given
by~\eqref{eqn:tail1} in Definition~\ref{def:scaling-constants}, and the
centering constants $c_n$ are given in
Definition~\ref{def:centering-constants} (see also
Remark~\ref{rem:centering-constants}).

\begin{rem}
  In the limit laws for quenched systems  that we obtain of type  (B) and (C), the centering sequence
  $c_n$ \emph{does not depend on the realization $\omega$}. This is in
  contrast to the case of the CLT, where a random centering is 
  necessary; see \cite[Theorem 9]{Abdelkader-Aimino2016} and
  \cite[Theorem 5.3]{NPT}.
\end{rem}

\section{A Poisson Point Process Approach to random and sequential
  dynamical systems}\label{sec:levy}

Our results are based on the Poisson point process approach developed by
Marta Tyran-Kami\'{n}ska \cite{TK,TK-dynamical} adapted to our random
setting (see Theorems~\ref{thm:seq_levy_FLT}
and~\ref{thm:seq_levy_was-TK-thm1.3}). Namely, convergence to a stable law
or a L\'evy process follows from the convergence of the corresponding
(Poisson) jump processes, and control of the small jumps.


A key role is played by Kallenberg's Theorem~\cite[Theorem 4.7]{Kal75}
to check convergence of the Poisson point processes,
$N_n\dto  N_{(\alpha)} $. Kallenberg's theorem does not assume stationarity
and hence we may use it in our setting.

In this section, we provide general conditions ensuring weak convergence to
L\'evy stable processes for non-stationary dynamical systems, following
closely the approach of Tyran-Kami\'nska \cite{TK-dynamical}. We start from
the very general setting of non-autonomous sequential dynamics and then
specialize to the case of quenched random dynamical systems, which will be
useful to treat iid random compositions in the later sections.

\subsection{Sequential transformations}\label{ssec:seq_levy}

Recall the notations introduced in Section~\ref{ssec:poisson_process}.
$(T_n)_{n \ge 1}$ is a sequence of measurable transformations on a probability space $(Y, \cB, \mu)$. For $n \ge 1$, recall we define
\[
T_1^n = T_n \circ \ldots \circ T_1.
\]


The proof of the following statement is essentially the same as the proof
of \cite[Theorem~1.1]{TK-dynamical}.

Note that the measure $\mu$ does not have to be invariant. Moreover
(see~\cite[Remark~2.1]{TK-dynamical}), the convergence
$X_n \dto X_{(\alpha)}$ holds even without the condition
$\mu(\phi \circ T_1^j \neq 0) = 1$, which is used only for the converse
implication of the ``if and only if''.

\begin{thm}[Functional stable limit law,
  {\cite[Theorem~1.1]{TK-dynamical}}]\label{thm:seq_levy_FLT}
  Let $\alpha\in (0,2)$ and suppose that $\mu(\phi \circ T_1^j \neq 0) = 1$
  for all $j\ge 0$. Then $X_n \dto X_{(\alpha)}$ in $\bD[0, \infty)$ under
  the probability measure $\mu$ for some constants $b_n>0$ and $c_n$ if and
  only if
  \begin{itemize}
  \item $N_n \dto N_{(\alpha)}$ and
  \item for all $\delta > 0$, $\ell \ge 1$, with $c_\alpha(\epsilon)$ given
    by \eqref{eq:c_alpha},
    \begin{equation}\label{eqn:seq_levy}
      \lim_{\epsilon \to 0} \limsup_{n
        \to \infty} \mu \left( \sup_{0 \le t \le \ell} \left|
          \frac{1}{b_n} \left[
          \sum_{j=0}^{\lfloor nt \rfloor - 1} \phi \circ T_1^j
          \one_{\left\{| \phi \circ T_1^j | \le \epsilon b_n \right\}}
          - t(c_n - b_n c_\alpha(\epsilon))\right] \right| \ge \delta \right) = 0
    \end{equation}
  \end{itemize}
\end{thm}


\begin{rem}
  In some cases the convergence $N_n \dto N_{(\alpha)}$ does not hold, but
  one has convergence of the marginals,
  $N_n((0,1] \times \cdot) \dto N_{(\alpha)}((0,1] \times \cdot)$. In this
  case, although unable to obtain a functional stable law convergence of
  type (C), we can in some settings prove the convergence to a stable law
  for the Birkhoff sums (convergence of type (B)).

  In particular, we are unable to prove $N_n^\omega \dto N_{(\alpha)}$ for
  the case of random intermittent maps. On the
  other hand, in the setting of random uniformly expanding maps  we use the spectral gap to show that
  $N_n^\omega \dto N_{(\alpha)}$, and then obtain the functional stable
  limit law.
\end{rem}

The next statement is \cite[Lemma 2.2, part (2)]{TK-dynamical}, which
follows from \cite[Theorem~3.2]{TK}. Again, the measure does not have to be
invariant.

\begin{thm}[Stable limit law, {\cite[Lemma
    2.2]{TK-dynamical}}]\label{thm:seq_levy_was-TK-thm1.3}

  For $\alpha\in(0,2)$, consider an observable $\phi$ on the probability
  measure $\mu$, and $c_\alpha(\epsilon)$ given by \eqref{eq:c_alpha}.

  If 
  \begin{equation}\nonumber 
    N_n ((0,1] \times \cdot )\dto N_{(\alpha)} ((0,1] \times\cdot)
  \end{equation}
  and, for all $\delta> 0$, 
  \begin{equation}\label{eqn:seq_levy_SLT}
    \lim_{\epsilon \to 0} \limsup_{n
      \to \infty} \mu \left( \left| \frac{1}{b_n} \left[ 
        \sum_{j=0}^{n - 1} \phi \circ T_1^j
        \one_{\left\{| \phi \circ T_1^j | \le \epsilon b_n \right\}}
        - (c_n - b_n c_\alpha(\epsilon)) \right] \right| \ge \delta \right) = 0 
  \end{equation}
  then
  \[
    \frac{1}{b_n}[\sum_{j=0}^{n-1}\phi\circ T_{1}^j - c_n] \dto
    X_{(\alpha)}(1)
  \]
  under the probability measure $\mu$.
\end{thm}

\subsection{Random dynamical systems}\label{ssec:iid}

Let $\phi : Y \to \bR$ be a measurable function such that
$\nu^\omega( \phi \neq 0) = 1$.
\begin{prop}[{\cite[proof of Theorem
    1.2]{TK-dynamical}}]\label{prop:alpha<1}\

  Let $\alpha \in (0,1)$. With $b_n$ as in
  Definition~\ref{def:scaling-constants} and $c_n=0$, suppose that for
  $\bP$-a.e.~$\omega \in \Omega$
  \begin{equation}\label{eqn:alpha<1}
    \lim_{\epsilon \to 0} \limsup_{n \to
      \infty} \frac{1}{b_n} \sum_{j=0}^{n \ell - 1} \bE_{\nu^{\sigma^j
        \omega}}( | \phi| \one_{\left\{ | \phi | \le \epsilon
        b_n\right\}}) = 0 \: \text{ for all } \ell \ge 1,
  \end{equation}
  and
  \[
    N_n^\omega \dto N_{(\alpha)}.
  \]
  Then $X_n^\omega \dto X_{(\alpha)}$ in $\bD[0, \infty)$ under the
  probability measure $\nu^\omega$ for $\bP$-a.e.~$\omega \in \Omega$.
\end{prop}

\begin{proof}
  We will check that the hypothesis of Theorem \ref{thm:seq_levy_FLT} are met
  for $\bP$-a.e.~$\omega$ with $T_n = T_{\sigma^{n-1} \omega}$,
  $\mu = \nu^\omega$. Recall that $c_n=c_\alpha(\epsilon) = 0$ when
  $\alpha \in (0,1)$. Using \cite[Theorem 1]{KW69} (see
  Theorem~\ref{thm.KW69}) and the equivariance of the family of measures
  $\left\{ \nu^\omega \right\}_{\omega \in \Omega}$, we have
\[
\nu^\omega \left( \sup_{0 \le t \le \ell} \left| \frac{1}{b_n}
    \sum_{j=0}^{\lfloor nt \rfloor - 1} \phi \circ T_\omega^j
    \one_{\left\{ | \phi \circ T_\omega^j  | \le \epsilon b_n
      \right\}} \right| \ge \delta \right)  \le \frac{1}{\delta b_n}
\sum_{j=0}^{n \ell - 1} \bE_{\nu^{\sigma^j \omega}} ( | \phi |
\one_{\left\{  |\phi | \le \epsilon b_n \right\}})
\]
which shows that condition \eqref{eqn:alpha<1} implies condition
\eqref{eqn:seq_levy} for all $\delta > 0$ and $\ell \ge 1$.
\end{proof}

\begin{rem}
  One could replace condition \eqref{eqn:alpha<1} by one similar to
  \eqref{eqn:ergodicity}, and use the argument in the proof of
  Proposition~\ref{prop:alpha_ge_1}.
\end{rem}

\begin{thm}[Kounias and Weng~{\cite[special
    case of Theorem 1 therein]{KW69}}]\label{thm.KW69}\

  Assume the random variables $X_k$ are in $L^1(\mu)$. Then
  \[
    \mu\left(\max_{1\le k \le n} \left|\sum_{\ell=1}^k X_\ell\right| \ge
      \delta\right) \le \frac{1}{\delta} \sum_{k=1}^n \bE_\mu(|X_k|).
  \]
\end{thm}

\begin{prop}\label{prop:alpha_ge_1}

  Let $\alpha \in [1, 2)$.

  With $b_n$ and $c_n$ as in Definitions~\ref{def:scaling-constants} and
  ~\ref{def:centering-constants}, and $c_\alpha(\epsilon)$ as in
  \eqref{eq:c_alpha}, suppose that for
  all $\epsilon > 0$ and all $\ell \ge 1$,
  \begin{equation}\label{eqn:alpha_ge_1}
    \lim_{n \to \infty} \sup_{0 \le t
      \le \ell} \left|  \frac{1}{b_n} \left[ \sum_{j=0}^{\lfloor nt \rfloor - 1}
      \bE_{\nu^{\sigma^j \omega}}(\phi \one_{\left\{| \phi | \le
          \epsilon b_n \right\}}) - t (c_n - b_n c_\alpha(\epsilon)) \right]\right| =
    0 \text{ \quad for $\bP$-a.e.~$\omega \in \Omega$,
    }
  \end{equation}
  and that for all $\delta > 0$
  %
\begin{multline}\label{eqn:ergodicity}
  \lim_{\epsilon \to 0} \limsup_{n \to
    \infty} 
  \underset{\omega \in \Omega}{\rm{esssup}}
  \, \nu^{\omega} \Big( \max_{1 \le k \le n} \Big| \frac{1}{b_n}
  \sum_{j=0}^{k-1} \big[ \phi \circ T_\omega^{j} \one_{\left\{ | \phi
      \circ T_\omega^j| \le \epsilon b_n \right\}} 
  - \bE_{\nu^{\sigma^j \omega}}( \phi \one_{\left\{ | \phi | \le
      \epsilon b_n\right\}})\big] \Big| \ge \delta \Big) =0.
\end{multline}

If $N_n^\omega \dto N_{(\alpha)}$ for $\bP$-a.e.~$\omega \in \Omega$, then
$X_n^\omega \dto X_{(\alpha)}$ in $\bD[0, \infty)$ under the probability
measure $\nu^\omega$ for $\bP$-a.e.~$\omega \in \Omega$.
\end{prop}

\begin{proof}
  As in the proof of Proposition \ref{prop:alpha<1}, we check the
  hypothesis of Theorem~\ref{thm:seq_levy_FLT} with
  $T_n = T_{\sigma^{n-1} \omega}$, $\mu = \nu^\omega$ for $\bP$-a.e.
  $\omega \in \Omega$. We will see that~\eqref{eqn:seq_levy} follows from
  \eqref{eqn:alpha_ge_1} and \eqref{eqn:ergodicity}.
  
  Using the equivariance of
  $\left\{ \nu^\omega \right\}_{\omega \in \Omega}$, we see that condition
  \eqref{eqn:seq_levy} is implied by~\eqref{eqn:alpha_ge_1}
  and~\eqref{eqn:condition} below:
  \begin{equation}\label{eqn:condition}
    \lim_{\epsilon \to 0} \limsup_{n \to \infty} \nu^\omega \left( \sup_{1
        \le k \le n \ell} \left| \frac{1}{b_n } \sum_{j=0}^{k - 1} \left[
          \phi \circ T_\omega^j \one_{\left\{| \phi \circ T_\omega^j
              | \le \epsilon b_n \right\}} - \bE_{\nu^{\sigma^j
              \omega}}(\phi \one_{\left\{ |\phi | \le \epsilon b_n
            \right\}})\right]\right| \ge \delta \right) = 0. 
  \end{equation}
  We next show that condition \eqref{eqn:ergodicity} implies
  \eqref{eqn:condition}.

  Since
\begin{multline*}
\left\{ \sup_{1 \le k \le n \ell} \left| \frac{1}{b_{n}} \sum_{j=0}^{k-1} \left[  \phi \circ T_\omega^j \one_{\left\{| \phi \circ T_\omega^j | \le \epsilon b_n \right\}} - \bE_{\nu^{\sigma^j \omega}}(\phi \one_{\left\{ |\phi | \le \epsilon b_n \right\}}) \right] \right| \ge \delta \right\}  \\
\subset \bigcup_{i=0}^{\ell -1} \left\{ \sup_{in < k \le (i+1)n} \left|\frac{1}{b_{n}} \sum_{j=in}^{k-1} \left[ \phi \circ T_\omega^j \one_{\left\{| \phi \circ T_\omega^j | \le \epsilon b_n \right\}} - \bE_{\nu^{\sigma^j \omega}}(\phi \one_{\left\{ |\phi | \le \epsilon b_n \right\}})\right] \right| \ge \frac{\delta}{\ell} \right\},
\end{multline*}
we obtain that, using again the equivariance, for $\bP$-a.e.~$\omega \in \Omega$, 

\begin{multline*}
\nu^\omega \left( \sup_{1 \le k \le n \ell} \left| \frac{1}{b_{n}} \sum_{j=0}^{k-1} \left[  \phi \circ T_\omega^j \one_{\left\{| \phi \circ T_\omega^j | \le \epsilon b_n \right\}} - \bE_{\nu^{\sigma^j \omega}}(\phi \one_{\left\{ |\phi | \le \epsilon b_n \right\}}) \right] \right| \ge \delta \right) \\
\le \sum_{i=0}^{\ell-1} \nu^{\sigma^{in} \omega} \left( \sup_{1 \le k \le n} \left|\frac{1}{b_{n}} \sum_{j=0}^{k-1} \left[  \phi \circ T_{\sigma^{in}\omega}^j \one_{\left\{| \phi \circ T_{\sigma^{in}\omega}^j | \le \epsilon b_n \right\}} - \bE_{\nu^{\sigma^j (\sigma^{in}\omega)}}(\phi \one_{\left\{ |\phi | \le \epsilon b_n \right\}}) \right] \right| \ge \frac{\delta}{\ell}\right) \\
\le \ell \, \cdot  \,\underset{\omega' \in \Omega}{\rm{esssup}} \, \nu^{\omega'} \left( \max_{1 \le k \le n} \left| \frac{1}{b_n} \sum_{j=0}^{k-1} \left[ \phi \circ T_{\omega'}^{j} \one_{\left\{ | \phi \circ T_{\omega'}^j| \le \epsilon b_n \right\}} - \bE_{\nu^{\sigma^j \omega'}}( \phi \one_{\left\{ | \phi | \le \epsilon b_n\right\}})\right] \right| \ge \frac{\delta}{\ell} \right).
\end{multline*}
Thus, condition \eqref{eqn:ergodicity} implies \eqref{eqn:condition}, which concludes the proof.
\end{proof}

The analogue for the convergence to a stable law is the following.

\begin{prop}\label{prop:quenched_stable_usual}
  Suppose that for $\bP$-a.e.~$\omega \in \Omega$, we have 
  \[
    N_n^\omega((0,1] \times \cdot) \dto N_{(\alpha)}((0,1] \times \cdot).
  \]

  If $\alpha \in (0,1)$ (so $c_n = 0$), we require in addition that
  \begin{equation}\label{eqn:seq_levy_01}
    \lim_{\epsilon \to 0} \limsup_{n \to \infty} \frac{1}{b_n}
    \sum_{j=0}^{n - 1} \bE_{\nu^{\sigma^j \omega}}( | \phi|
    \mathds{1}_{\left\{ | \phi | \le \epsilon b_n\right\}}) = 0
  \end{equation}

  If $\alpha \in [1,2)$, we require instead of~\eqref{eqn:seq_levy_01}
  that for all $\epsilon >0$,
  \[
    \lim_{n \to \infty} \left| \frac{1}{b_n} \left[\sum_{j=0}^{n - 1}
      \bE_{\nu^{\sigma^j \omega}}( \phi \mathds{1}_{\left\{ | \phi | \le
          \epsilon b_n\right\}}) - (c_n - b_n c_\alpha(\epsilon)) \right]\right|= 0
  \]
  and
  \[
    \lim_{\epsilon \to 0} \limsup_{n \to \infty} \nu^{\omega} \left( \left|
        \frac{1}{b_n} \sum_{j=0}^{n-1} \left[ \phi \circ T_\omega^{j}
          \mathds{1}_{\left\{ | \phi \circ T_\omega^j| \le \epsilon b_n
            \right\}} - \bE_{\nu^{\sigma^j \omega}}( \phi
          \mathds{1}_{\left\{ | \phi | \le \epsilon b_n\right\}})\right]
      \right| \ge \delta \right) =0.
  \]

  Then
  \[
    \frac{1}{b_n}\left[\sum_{j=0}^{n-1}\phi\circ T_{\omega}^j - c_n \right]\dto
    X_{(\alpha)}(1)
  \]
  under the probability measure $\nu^\omega$ for $\bP$-a.e.
  $\omega \in \Omega$.
\end{prop}

\begin{proof}
  We check the conditions of Theorem~\ref{thm:seq_levy_was-TK-thm1.3}.

  The proof for $\alpha\in(0,1)$ is similar to the proof of
  Proposition~\ref{prop:alpha<1}, the proof of the case $\alpha\in[1,2)$ is
  similar to the proof of Proposition~\ref{prop:alpha_ge_1}.
\end{proof}


\subsection{The annealed transfer operator}

We assume that the random dynamical system
$F:\Omega \times [0,1]\to \Omega \times [0,1]$, 
\[
F(\omega, x) =(\sigma \omega, T_\omega (x)) 
\]
which can also be viewed as
a Markov process on $[0,1]$, has a stationary measure $\nu$ with density
$h$. The map $F: \Omega \times [0,1] \to \Omega \times [0,1]$ will preserve
$\bP\times \nu$. Recall that $\bP:=\{(p_1, \dots, p_m)\}^\Z$.

We use the notation $P_{\mu,i}$ for the transfer operator of
$T_i:[0,1]\to[0,1]$ with respect to a measure $\mu$ on $[0,1]$, i.e.
\[
  \int f \cdot g\circ T_i d\mu =\int (P_{\mu,i} f) g d\mu, \text{ for all
    $f\in L^{1} (\mu)$, $g\in L^{\infty} (\mu)$}.
\]

The annealed transfer operator is defined by
\[
  P_{\mu} (f):=\sum_{i=1}^{m}p_i P_{\mu,i} (f)
\]
 with adjoint
\[
  U(f):=\sum_{i=1}^m p_i f\circ T_{i}
\]
which satisfies the duality relation
\[
  \int f (g\circ U) d\mu =\int (P_{\mu} f) g d\mu, \text{ for all
    $f\in L^{1} (\mu)$, $g\in L^{\infty}(\mu)$}.
\]

As above, we assume there are sample measures
$d \nu^\omega = h_{\omega} d x $ on each fiber $[0,1]$ of the skew product
such that
\[
  P_{\omega} h_{\omega}=h_{\sigma \omega}
\]
where $P_{\omega}$ is the transfer operator of $T_{\omega_0}$ with respect
to the Lebesgue measure.

Therefore
\[
\nu (A)=\int_{\Omega} [\int_A h_{\omega} dx] d\bP(\omega)
\]
for all Borel sets $A\subset [0,1]$.


\subsection{Decay of correlations}

We now consider the decay of correlations properties of the annealed
systems associated to maps satisfying (LY), (Dec) and (Min) and  intermittent maps.

By \cite[Proposition 3.1]{ANV15} in the
setting of maps satisfying  (LY), (Dec) and (Min) we have exponential decay in $BV$ against
$L^1$: there are $C>0$, $0<\lambda<1$ such that
\[
  \left| \int f g\circ U^n d\nu -\int f d\nu \int g d\nu \right|\le C \lambda^n
  \|f\|_{BV} \|g\|_{L^1 (\nu)}
\]

In the setting of intermittent maps, by \cite[Theorem
1.2]{Bahsoun:2016aa}, we have polynomial decay in H\"older against
$L^{\infty}$: there exists $C>0$ such that
\[
  \left| \int f g\circ U^n d\nu -\int f d\nu \int g d\nu \right|\le C
  n^{1-\frac{1}{\gamma_{min}}} \|f\|_{\text{H\"older}} \|g\|_{L^{\infty}
    (\nu)}.
\]

We now consider a useful property satisfied by  our class of random uniformly expanding maps. 

\begin{defin}[{\bf Condition U}]\label{def:conditionU}
  We assume that almost each $\nu^\omega$ is absolutely continuous with
  respect to the Lebesgue measure $m$, and
    \begin{align}\label{eqn:density_bounded}
      &\text{for some $C >0$,\qquad } \bP\text{-a.e.~} \omega \in \Omega
        \implies C^{-1 } \le
        h_\omega:=\frac{d\nu^\omega}{d
        \nu} \le C, \: m {\text{-a.e.}}\\
      \label{eqn:density_holder}
      & \text{the map } \omega \in \Omega \mapsto h_\omega \in L^\infty(m)
        \text{ is H\"older continuous.}
    \end{align}
    Consequently, the stationary measure $\nu$ is also absolutely
    continuous with respect to $m$, with density $h \in L^\infty(m)$ given
    by $h(x) = \int_\Omega h_\omega(x) \bP(d\omega)$ and satisfying
    \eqref{eqn:density_bounded}.
\end{defin}

\begin{lem}
  Properties {\bf (LY)}, {\bf (Min)} and {\bf (Dec)} imply Condition U. Namely,
  there exists a unique H\"older map
  $\omega \in \Omega \mapsto h_\omega \in {\rm BV}$ such that
  $P_\omega h_\omega = h_{\sigma \omega}$ and \eqref{eqn:density_bounded},
  \eqref{eqn:density_holder} are satisfied by~\cite{ANV15}.
\end{lem}

\begin{proof}
  By {\bf (Dec)}, and as all the operators $P_\omega$ are Markov
  with respect to $m$, we have
  \begin{equation} \label{eqn:cauchy} \|P_{\sigma^{-(n+k)}\omega}^{n+k}
    \one - P_{\sigma^{-n} \omega}^n \one \|_{\rm BV} \le C
    \kappa^n \|\one - P_{\sigma^{-(n+k)}\omega}^k \one \|_{\rm
      BV} \le C \kappa^n,
  \end{equation}
  which proves that $(P_{\sigma^{-n} \omega}^n \one)_{n \ge 0}$ is a
  Cauchy sequence in BV converging to a unique limit $h_\omega \in BV$
  satisfying $P_\omega h_\omega = h_{\sigma \omega}$ for all $\omega$. The
  lower bound in \eqref{eqn:density_bounded} follows from the condition
  {\bf (Min)}, while the upper bound is a consequence of the uniform
  Lasota-Yorke inequality {\bf (LY)}, as actually the family
  $\left\{h_\omega\right\}_{\omega \in \Omega}$ is bounded in BV. To prove
  the H\"older continuity of $\omega \mapsto h_\omega$ with respect to the
  distance $d_\theta$, we remark that if $\omega$ and $\omega'$ agree in
  coordinates $|k| \le n$, then
  \[
    \| h_\omega - h_{\omega'}\|_{\rm BV} =
    \|P_{\sigma^{-k}\omega}^k(h_{\sigma^{-k}\omega} - h_{\sigma^{-k}
      \omega'})\|_{\rm BV} \le C \theta^n \le C d_\theta(\omega, \omega').
  \]
\end{proof}

 \begin{rem}
  Note that the density $h$ of the stationary measure $\nu$ also belongs to
  BV and is uniformly bounded from above and below, as the average of
  $h_\omega$ over $\Omega$.
\end{rem}


\subsubsection{The sample measures $h_\omega$}

The regularity properties of the sample measures $h_{\omega}$, both as functions of $\omega$ and as 
functions of $x$ on $[0,1]$ play a key role in our estimates. 
We will first recall how the sample measures are constructed. Suppose
$\omega:=(\ldots, \omega_{-1}, \omega_0,\omega_1,\ldots, \omega_n,\ldots,)$
and define $h_n (\omega)=P_{\omega_{-1}}\ldots P_{\omega_{-n} }1$ as a
sequence of functions on the fiber $I$ above $\omega$. In the setting both of random uniformly expanding
maps  and of  intermittent maps $\{h_n (\omega) \}$ is
a Cauchy sequence and has a limit $h_{\omega}$.

In  the  setting of  random expanding maps,   $h_\omega$  is uniformly  $BV$  in
$\omega$ as
\[
  \|h_n (\omega)-h_{n+1} (\omega)\|_{BV} \le \|P_{\omega_{-1}}
  P_{\omega_{-2}} \ldots P_{\omega_{-n}} (1-P_{\omega_{-n-1} }1 )\|_{BV}
  \le C \lambda^{n}.
\]

In the setting of intermittent maps with
$\gamma_{max}=\max_{1\le i \le m} \{\gamma_i\}$, the densities $h_\omega$
lie in the cone
\[
  \begin{array}{ll} L:=\big\{
                  f\in \mathcal{C}^0((0,1])\cap L^1(m), &
                f\geq 0,\ f \text{ non-increasing}, \\
                & X^{\gamma_{max}+1}f \text{ increasing},
                 f(x)\leq ax^{-\gamma_{max}}m(f)\big\}\end{array}
\] 
where $X(x) = x$ is the identity function and $m(f)$ is the integral of $f$
with respect to $m$. In \cite{AimHuNicTor_2015} it is proven that for a
fixed value of $\gamma_{max} \in(0,1)$, provided that the constant $a$ is
big enough, the cone $L$ is invariant under the action of all transfer
operators $P_{\gamma_i}$ with $0<\gamma_i \leq\gamma_{max}$ and so (see
e.g.~\cite[Proposition 3.3]{NPT}, which summarizes results
of~\cite{NTV18})
\[
  \|h_n (\omega)-h_{n+k} (\omega)\|_{L^1(m)} \le \|P_{\omega_{-1}}
  P_{\omega_{-2}} \ldots P_{\omega_{-n}} (1-P_{\omega_{-n-1} } \dots
  P_{\omega_{-n-k} }1 )\|_{L^1(m)}
  \]
  \[
   \le C_{\gamma_{max} }
  n^{1-\frac{1}{\gamma_{max}}} (\log n)^{\frac{1}{\gamma_{max}}}
\]
whence $h_\omega\in L^1(m)$. In later arguments we will use the
approximation
\begin{equation}\label{intermittent_approx}
  \|h_n(\omega)-h_\omega\|_{L^1(m)} \le C_{\gamma_{max} }
  n^{1-\frac{1}{\gamma_{max}}}
  (\log n)^{\frac{1}{\gamma_{max}}}.
\end{equation}
We mention also the recent paper~\cite{Korepanov_Leppanen} where the
logarithm term in Equation~\eqref{intermittent_approx} is shown to be
unnecessary and moment estimates are given.

We now show that $h_{\omega}$ is a H\"older function of $\omega$ on
$(\Omega, d_{\theta})$ in the setting of random expanding maps.

For $\theta \in (0,1)$, we introduce on $\Omega$ the symbolic metric 
\[
  d_\theta(\omega, \omega') = \theta^{s(\omega, \omega')}
\]
where
$s(\omega, \omega') = \inf \left\{k \ge 0 \, : \, \omega_\ell \neq
  \omega_\ell' \text{ for some } |\ell|\le k\right\}$.


Suppose $\omega$, $\omega'$ agree in coordinates $|k| \le n$ (i.e.
backwards and forwards in time) so that
$d_{\theta} (\omega, \omega^{'}) \le \theta^n$ in the symbolic metric on
$\Omega$. Then
\[
  \|h_{\omega}-h_{\omega'}\|_{\rm BV} \le \|P_{\omega_{-1}} P_{\omega_1} \ldots
  P_{\omega_{-n+1}}(h_{(\sigma^{-n+1} \omega)}-h_{(\sigma^{-n+1}
    \omega')})\|_{\rm BV} 
    \]
    \[
     \le C \lambda^{n-1} = C' d_{\theta}
  (\omega,\omega')^{\log_\theta \lambda}
\]
Recall that $\|f\|_{\infty} \le C \|f\|_{\rm BV}$, see e.g. \cite[Lemma
2.3.1]{BG97}.

That is, Condition U (see Definition~\ref{def:conditionU}) holds for
random expanding maps.

The map $\omega \mapsto h_{\omega}$ is not H\"older in the setting of
intermittent maps; in several arguments we will use the regularity
properties of the approximation $h_n(\omega)$ for $h_\omega$.

However, on intervals that stay away from zero, all functions in the cone
$L$ are comparable to their mean. Therefore, on sets that are uniformly
away from zero, all the above densities/measures ($d \nu = h dx$,
$h_\omega$, $h_n(\omega)$) are still comparable.

Namely,
\begin{equation}\label{eq:intermittent-densities-comparable}
  \begin{aligned} \text{ for any } \delta \in (0,1) \text{ there is }
C_\delta > 0 \text{ such that }\\ h \in L \implies 1/C_\delta < h(x)/m(h) <
C_\delta \text{ for } x\in [\delta, 1]
  \end{aligned}
\end{equation} Indeed, $h/m(h)$ is bounded below by~\cite[Lemma
2.4]{LivSauVai}, and the upper bound follows from the definition of the
cone.

\section{Ancilliary  Results}\label{sec:Ancilliary_Results}

Let $x_0 \in [0,1]$, and, for $\alpha \in (0,2)$, recall we define the function
$\phistar(x) = |x - x_0|^{ - \frac{1}{\alpha}}$.
It is easy to see that $\phistar$ is regularity varying with index $\alpha$ and that $p=1$.



\subsection{Exponential law and point process results} \label{ssec:exp_process}

We denote by $\cJ$ the family of all finite unions of intervals of the form
$(x, y]$, where $-\infty \le x < y \le \infty$ and $0 \notin [x,y]$.


For a measurable subset $U \subset [0,1]$, we define the hitting time of
$(\omega, x) \in \Omega \times [0,1]$ to $U$ by
\begin{equation} \label{eqn:hitting_times}
R_U(\omega)(x) := \inf \left\{ k \ge 1 \, : \, T_\omega^k(x) \in U \right\}.
\end{equation} 

Recall that $\phistar(x):=d(x,x_0)^{-\frac{1}{\alpha}}$ depends on the
choice of $x_0 \in [0,1]$. Recall also that
\[
  \cD = \cup_{n \ge 0} \cup_{\omega \in \Omega} \partial \cA_\omega^n
\]
the set of discontinuities of all the maps $T_\omega^n$.

\begin{thm}\label{thm:return}
  In the setting of Section~\ref{sec:expanding_iid}, 
  assume {\bf (LY)}, {\bf (Min)} and {\bf (Dec)}. If $x_0 \notin \cD$ is non-recurrent,
  then, for $\bP$-a.e.~$\omega \in \Omega$ and all $0 \le s < t$,
  \[
    \lim_{n \to \infty} \nu^{\sigma^{\lfloor ns \rfloor}\omega}
    \left(R_{A_n}(\sigma^{\lfloor ns \rfloor}\omega) > \lfloor
      n(t-s) \rfloor\right) = e^{-(t-s) \Pi_\alpha(J)}.
  \]
  where $A_n := \phistar^{-1}(b_n J)$, $J \in \cJ$.
\end{thm}

\begin{thm}\label{thm:return_intermittent}
  In the setting of intermittent maps assume that
  $\gamma_{max} <\frac{1}{3}$. Then for $m$-a.e.~$x_0$ for $\bP$-a.e.
  $\omega \in \Omega$ and all $0 \le s < t$,
  \[
    \lim_{n \to \infty} \nu^{\sigma^{\lfloor ns \rfloor}\omega}
    \left(R_{A_n}(\sigma^{\lfloor ns \rfloor}\omega) > \lfloor
      n(t-s) \rfloor\right) = e^{-(t-s) \Pi_\alpha(J)}.
  \]
  where $A_n := \phistar^{-1}(b_n J)$, $J \in \cJ$.
\end{thm}

\begin{thm}\label{thm:poisson_expanding}

  In the setting of Section~\ref{sec:expanding_iid}, 
  assume {\bf (LY)}, {\bf (Min)} and {\bf (Dec)}. If $x_0 \notin \cD$ is non-recurrent,
  then for $\bP$-a.e.~$\omega \in \Omega$, then
  \[
    N_n^\omega \dto N_{(\alpha)},
  \]
  under the probability $\nu^\omega$.
\end{thm}

\begin{thm}\label{intermittent_poisson}

 In the setting of intermittent maps for $m$-a.e.~$x_0$
  for $\bP$-a.e.~$\omega$,
  \[
    N^{\omega}_n ((0,1]\times \cdot) \dto  N_{(\alpha)} ((0,1] \times
    \cdot)
  \]
  \end{thm}

After some preliminary lemmas and results Theorem~\ref{thm:return} is proved in Section~8.1, Theorem~\ref{thm:return_intermittent} in
Section~8.2, Theorem~\ref{thm:poisson_expanding} in Section~9.1 and Theorem~\ref{intermittent_poisson} in Section~9.2.

\section{Scheme of proofs}

\subsection{Two useful lemmas}

We now proceed to the proofs of the main results. We will use the following
technical propositions which are a form of spatial ergodic theorem which
allows us to prove exponential and Poisson limit laws.

\begin{lem}\label{lemma:tech} 
  Assume Condition U and let $\chi_n : Y \to \bR$ be a sequence
  of functions in $L^1(m)$ such that
  $\bE_m(| \chi_n |) = \cO(n^{-1} \widetilde{L}(n))$ for some slowly
  varying function $\widetilde{L}$. Then, for $\bP$-a.e.
  $\omega \in \Omega$ and for all $\ell \ge 1$,
  \[
    \lim_{n \to \infty} \sup_{0 \le k \le \ell} \left| \sum_{j=0}^{kn - 1}
      \left(\bE_{\nu^{\sigma^j \omega}}( \chi_n) - \bE_\nu(\chi_n)
      \right)\right| = 0.
  \]

  Therefore, given $( s, t] \subset [0, \infty)$ and $\epsilon > 0$, for
  $\bP$-a.e.~$\omega$ there exists $N(\omega)$ such that
  \begin{equation}\nonumber
    \left|\sum_{r=\lfloor n s \rfloor +1}^{\lfloor n t \rfloor} \left(\bE_{\nu^{\sigma^j \omega}}( \chi_n) -
        \bE_\nu(\chi_n)\right)\right|\le \epsilon
  \end{equation}
  for all $n \ge N(\omega)$.
\end{lem}

\begin{proof}

  We obtain the second claim by taking the difference between two values of
  $\ell$ in the first claim.

Fix $\ell \ge 1$. For $\delta >0$, let 
\[
U_k^n(\delta) = \left\{  \omega \in \Omega \, : \, \left| \sum_{j=0}^{kn  -1} \left(\bE_{\nu^{\sigma^j \omega}}(\chi_n) - \bE_\nu(\chi_n) \right) \right| \ge \delta \right\},
\]
and
\[
B^n(\delta) = \left\{  \omega \in \Omega \, : \, \sup_{0 \le k \le \ell} \left| \sum_{j=0}^{kn  -1} \left(\bE_{\nu^{\sigma^j \omega}}(\chi_n) - \bE_\nu(\chi_n) \right) \right| \ge \delta \right\}.
\]
Note that 
\[
B^n(\delta) = \bigcup_{k=0}^\ell U_k^n(\delta).
\]
We define $f_n(\omega) = \bE_{\nu^\omega}(\chi_n)$ and
$\overline{f}_n = \bE_{\bP}(f_n)$. We claim that $f_n:\Omega\to\R$ is
H\"older with norm $\|f_n \|_\theta = \cO(n^{-1}\widetilde{L}(n))$. Indeed,
for $\omega \in \Omega$, we have
\[
|f_n(\omega)| = \left| \int_Y \chi_n(x) d \nu^\omega(x) \right| \le \|h_\omega\|_{L^\infty_m} \|\chi_n\|_{L^1_m} \le \frac{C}{n}\widetilde{L}(n),
\]
and for $\omega, \omega' \in \Omega$, we have
\begin{align*}
|f_n(\omega)  - f_n(\omega')| &=  \left|\int_Y \chi_n(x) d \nu^\omega(x) - \int_Y \chi_n(x) d\nu^{\omega'}(x) \right| \\
& \le \int_Y | \chi_n(x) | \cdot  |h_\omega(x) - h_{\omega'}(x)| dm(x) \\
& \le \| h_\omega - h_{\omega'}\|_{L^\infty_m} \| \chi_n \|_{L^1_m} \\
& \le \frac C n \widetilde{L}(n) d_\theta(\omega, \omega'),
\end{align*}
since $\omega \in \Omega \mapsto h_\omega \in L^\infty(m)$ is H\"older continuous. In particular, we also have that $\overline{f}_n = \cO(n^{-1}\widetilde{L}(n))$.

We have, using Chebyshev's inequality,
\begin{align*}
\bP(U_k^n(\delta)) &= \bP\left( \left\{ \omega \in \Omega \, : \, \left|\sum_{j=0}^{kn -1} \left( f_n \circ \sigma^j - \overline{f}_n \right) \right| \ge \delta \right\}\right) \\
& \le \frac{1}{\delta^2} \bE_{\bP}\left(\left(\sum_{j=0}^{kn -1} \left( f_n \circ \sigma^j - \overline{f}_n \right)\right)^2 \right) \\
& \le \frac{1}{\delta^2} \left[ \sum_{j=0}^{kn - 1} (\bE_{\bP} |f_n \circ \sigma^j -  \overline{f}_n|^2 + 2 \sum_{0 \le i < j \le kn - 1} \bE_{\bP}((f_n \circ \sigma^i - \overline{f}_n)(f_n \circ \sigma^j - \overline{f}_n)) \right].
\end{align*}
By the $\sigma$-invariance of $\bP$, we have
\[
\bE_{\bP} |f_n \circ \sigma^j -  \overline{f}_n|^2 = \bE_{\bP} |f_n  -  \overline{f}_n|^2,
\]
and, since $(\Omega, \bP, \sigma)$ admits exponential decay of correlations for H\"older observables, there exist $\lambda \in (0,1)$ and $C >0$ such that
\begin{align*}
\bE_{\bP}((f_n \circ \sigma^i - \overline{f}_n)(f_n \circ \sigma^j - \overline{f}_n)) &= \bE_{\bP}((f_n - \overline{f}_n)(f_n \circ \sigma^{j-i}- \overline{f}_n)) \\
& \le C \lambda^{j-i} \| f_n - \overline{f}_n\|_\theta^2.
\end{align*}
We then obtain that 
\begin{align*}
\bP(U_k^n(\delta)) &\le \frac{C}{\delta^2} \left[kn \|f_n - \overline{f}_n\|_{L^2_m}^2 + 2 \sum_{0 \le i < j \le kn - 1} \lambda^{j-i}  \| f_n - \overline{f}_n\|_\theta^2 \right]\\
& \le C \frac{nk}{\delta^2} \|f_n\|_\theta^2 \\
&\le C \frac{k}{n \delta^2}(\widetilde{L}(n))^2,
\end{align*}
which implies that
\[
\bP(B^n(\delta)) \le C \frac{\ell^2}{n \delta^2} (\widetilde{L}(n))^2.
\]

Let $\eta >0$. By the Borel-Cantelli lemma, it follows that for $\bP$-a.e.
$\omega \in \Omega$, there exists $N(\omega, \delta) \ge 1$ such that
$\omega \notin B^{\lfloor p^{1+\eta}\rfloor} (\delta)$ for all
$p \ge N(\omega, \delta)$.

Let now
$P:= \lfloor p^{1+\eta} \rfloor < n \le P' = \lfloor (p+1)^{1+\eta}
\rfloor$ for $p$ large enough. Let $0 \le k \le \ell$. Then, since
$\|f_n\|_\infty = \cO(n^{-1}\widetilde{L}(n))$,
\begin{align*}
\left| \sum_{j=0}^{kP -1 } \left(f_n(\sigma^j \omega) - \overline{f}_n \right) - \sum_{j=0}^{kn - 1}\left(f_n(\sigma^j \omega) - \overline{f}_n \right)\right|  &\le \sum_{j=kP}^{kn - 1} \left|f_n(\sigma^j \omega) - \overline{f}_n \right| \\
& \le C \frac{P' - P}{P} \widetilde{L}(n) \le C  \frac{\widetilde{L}(p^{1 + \eta})}{p},
\end{align*}
because on the one hand
\[
\frac{P' - P}{P} = \frac{\lfloor (p+1)^{1+\eta} \rfloor- \lfloor p^{1+\eta}\rfloor  }{\lfloor p^{1+\eta} \rfloor} = \cO \left( \frac 1 p \right),
\]
and on the other hand, by Potter's bounds, for $\tau > 0$,
\[
\widetilde{L}(n) \le C \widetilde{L}(P) \left( \frac{n}{P} \right)^\tau \le C \widetilde{L}(P) \left( \frac{P'}{P} \right)^\tau \le C\widetilde{L}(P). 
\]

Since 
\[
\left| \sum_{j=0}^{kP  -1} \left(f_n(\sigma^j \omega) - \overline{f}_n \right) \right| < \delta
\]
for all $0 \le k \le \ell$, it follows that for $\bP$-a.e.~$\omega$, there exists $N(\omega, \delta)$ such that $\omega \notin B^n(2 \delta)$ for all $n \ge N(\omega, \delta)$, which concludes the proof.
\end{proof}

We now consider a corresponding result to Lemma~\ref{lemma:tech} in
the setting of intermittent maps.

\begin{lem}\label{lemma:tech_intermittent}
  Assume that
  $\gamma_{max} < 1/2$, and that $\chi_n \in L^1(m)$ is such that
  $\bE_m( |\chi_n|) = \cO(n^{-1})$, $\|\chi_n\|_\infty = \cO(1)$ and there
  is $\delta >0$ such that ${\rm supp}(\chi_n) \subset [\delta,1]$ for all
  $n$.

Then, for $\bP$-a.e.~$\omega \in \Omega$ and for all $\ell \ge 1$, 
\[
  \lim_{n \to \infty} \sup_{0 \le k \le \ell} \left| \sum_{j=0}^{kn - 1}
    \left(\bE_{\nu^{\sigma^j \omega}}( \chi_n) - \bE_\nu(\chi_n)
    \right)\right| = 0.
\]

\end{lem}

\begin{proof}
  In the setting of intermittent maps we must modify the argument
  of Lemma~\ref{lemma:tech} slightly as $h_\omega$ is not a H\"older
  function of $\omega$. Instead, we consider
  $h^i_{\omega}=P^i_{\sigma^{-i}\omega}\mathds{1}$.
and use that, by \eqref{intermittent_approx},
\begin{equation}\label{eq:intermittent-density-convergence}
  \|h^i_{\omega}-h_{\omega}\|_{L^{1}(m)} \le C
  i^{1-\frac{1}{\gamma_{max}}} 
  \text { \qquad (leaving out the log term).}
\end{equation}
Note that $h^i_{\omega}$ is the $i$-th approximate to
$h_{\omega}$ in the pullback construction of $h_{\omega}$. Let $\nu_\omega^i$ be the measure such that $\frac{d \nu_\omega^i}{dm} = h_\omega^i$.


Consider
\[
f_n^i(\omega) = \bE_{\nu_\omega^i}(\chi_n)\:,\qquad f_n(\omega) = \bE_{\nu^\omega}(\chi_n)
\]
\[
\overline{f}_n^i = \bE_\bP(f_n^i)\:,\qquad \overline{f}_n  = \bE_\bP(f_n).
\]
By~\eqref{eq:intermittent-densities-comparable}, on the set $[\delta, 1]$
the densities involved ($h_\omega^k, h_\omega, h=d \nu/d m$) are uniformly
bounded above and away from zero. Thus
$\|f_n^i \|_\infty = \cO(n^{-1})$.

Pick $0<a<1$ is such that $\beta:=(\frac{1}{\gamma_{max}}-1)a -1 >0$.

For a given $n$ take $i=i_n=n^a$.
By~\eqref{eq:intermittent-density-convergence}, for all $\omega$, $n$ and
$i=n^a$
\[
  |f_{n}^i (\omega) - f_{n} (\omega)| \le \|h_\omega^i -h_\omega\|_{L^1(m)} \| \chi_n \|_{L^\infty(m)} = \cO(n^{-(\beta+1)}).
\]

Then
\[
  |\overline{f}_n^i -\overline{f}_n| = \cO(n^{-(\beta+1)})
\]
and 
\[
  \left|\sum_{r=0}^{kn - 1} [f_{n}^i (\sigma^r \omega)-f_n (\sigma^r
    \omega)] \right|\le C \ell n^{-\beta}.
\]

Given $\epsilon$, choose $n$ large enough that for all $0 \le k \le \ell$,
\[
  \left\{ \omega \in \Omega \, : \, \left| \sum_{r=0}^{kn -1} (f_n
  (\sigma^r \omega)- \overline{f}_n) \right| >\epsilon\right\}
  \subset \left\{ \omega \in \Omega \, : \, \left| \sum_{r=0}^{kn-1}
  (f^i_n (\sigma^r \omega)-\overline{f}^i_n) \right| >\frac{\epsilon}{2}\right\}.
\]
By Chebyshev
\begin{align*}
  \bP \left(\ \left| \sum_{r=0}^{kn-1}
  (f^i_n \circ \sigma^r-\overline{f}^i_n)\right| >\frac{\epsilon}{2}\right) 
& \le 
 \frac{4}{\epsilon^2} \sum_{r=0}^{kn-1}
  \bE_{\bP} \left( \left[f_n^i \circ \sigma^r -\overline{f}^i_n\right]^2  \right)\\
 & +\frac{4}{\epsilon^2} \left[2 \sum_{r=0}^{kn-1}
  \sum_{u=r+1}^{kn-1} \left| \bE_{\bP}[(f^i_n \circ \sigma^r -\overline{f}^i_n)(f_n^i
    \circ\sigma^u-\overline{f}^i_n)]\right|\right]
\end{align*}
We bound 
\[
  \sum_{r=0}^{kn-1}
  \bE_{\bP} \left( \left[f_n^i -\overline{f}^i_n\right]^2  \right) \le C \sum_{r=0}^{kn-1} \|f_n^i\|_\infty^2 \le \frac{C \ell}{ n}
  \]
  and note that if $|r-u|>n^{a}$ then by independence
  \[
    \bE_{\bP} \left[(f^i_n \circ\sigma^r -\overline{f}^i_n)(f_n \circ\sigma^u -\overline{f}^i_n)\right]=
    \bE_{\bP}\left[f^i_n \circ\sigma^r - \overline{f}^i_n\right] \bE_{\bP}\left[f^i_n \circ\sigma^u
 - \overline{f}^i_n\right]=0
    \]
    and hence we may bound
\[
  \sum_{r=0}^{kn-1}
  \sum_{u=r+1}^{kn-1} \left| \bE_{\bP}[(f^i_n \circ \sigma^r -\overline{f}^i_n)(f_n^i
    \circ\sigma^u-\overline{f}^i_n)]\right| \le \frac{C \ell}{n^{1-a}}.
\]
Thus, for $n$ large enough,
\[
  \bP \left( \left\{ \omega \in \Omega \, : \, \left| \sum_{r=0}^{kn-1} [f_n (\sigma^r
  \omega)-\overline{f}_n]\right| >\epsilon \right\}\right) \le \frac{C\ell}{n^{1-a}\epsilon^2} .
\]
The rest of the argument proceeds as in the case of Lemma~\ref{lemma:tech}
using a speedup along a sequence $n=p^{1+\eta}$ where
$\eta>\frac{a}{1-a}$, since $\|f_n\|_\infty = \cO(n^{-1})$ still holds.
\end{proof}

\subsection{Criteria for stable laws and functional limit laws}

The next theorem shows that for regularly varying observables, Poisson
convergence and Condition $U$ imply convergence in the $J_1$ topology if
$\alpha \in (0,1)$ and gives an additional condition to be verified in the
case $\alpha \in [1,2)$.

Note that \eqref{eqn:ergodicity_iid} is essentially
condition~\eqref{eqn:ergodicity} of Proposition~\ref{prop:alpha_ge_1}.

\begin{thm}\label{thm:iid_main}
  Assume $\phi$ is regularly varying,  Condition $U$ holds and that
 \[
   N_n^\omega \dto N_{(\alpha)}
 \]
 for $\bP$-a.e.~$\omega \in \Omega$. 
  
  If $\alpha \in [1, 2)$, assume furthermore that for all $\delta > 0$, and $\bP$-a.e.~$\omega\in \Omega$
  \begin{equation}\label{eqn:ergodicity_iid}
    \lim_{\epsilon \to 0} \limsup_{n \to \infty} \nu \left( \max_{1 \le k
        \le n} \left| \frac{1}{b_n} \sum_{j=0}^{k-1} \left[ \phi \circ
          T_\omega^{j} \one_{\left\{ | \phi \circ T_\omega^j| \le
              \epsilon b_n \right\}} - \bE_{\nu^{\sigma^j \omega}}( \phi
          \one_{\left\{ | \phi | \le \epsilon b_n\right\}})\right]
      \right| \ge \delta \right) =0. 
  \end{equation}

  Then $X_n^\omega \dto X_{(\alpha)}$ in $\bD[0, \infty)$ under the
  probability measure $\nu^\omega$ for $\bP$-a.e.~$\omega \in \Omega$.
\end{thm}

\begin{rem}
  From \eqref{eqn:density_bounded} and Theorem~\ref{thm:seq_levy_FLT}, it
  follows that the convergence of $X_n^\omega$ also holds under the
  probability measure $\nu$.
\end{rem}


\begin{proof}[Proof of Theorem \ref{thm:iid_main}]

  When $\alpha \in (0,1)$, we check the hypothesis of Proposition
  \ref{prop:alpha<1}. Using \eqref{eqn:density_bounded}, we have
\[
  \left| \frac{1}{b_n} \sum_{j=0}^{n \ell - 1} \bE_{\nu^{\sigma^j \omega}}(
    | \phi| \one_{\left\{ | \phi | \le \epsilon b_n\right\}}) \right|
  \le C \frac{n \ell}{b_n} \bE_\nu(| \phi| \one_{\left\{ |\phi| \le
      \epsilon b_n\right\}})
\]
Using Proposition \ref{prop:karamata}, we see that
condition \eqref{eqn:alpha<1} is satisfied since $\alpha <1$, thus proving
the theorem in this case.

When $\alpha \in [1,2)$, we consider instead
Proposition~\ref{prop:alpha_ge_1}. Firstly, we remark that condition
\eqref{eqn:ergodicity} is implied by \eqref{eqn:ergodicity_iid} and
\eqref{eqn:density_bounded}. It remains to check
condition~\eqref{eqn:alpha_ge_1}, which constitutes the rest of the proof.

\medskip

If $\alpha \in (1,2)$ we have

\begin{equation} \label{eqn:iid_main} \left| \frac{1}{b_n} \left[
    \sum_{j=0}^{\lfloor nt \rfloor - 1} \bE_{\nu^{\sigma^j \omega}}(\phi
    \one_{\left\{| \phi | \le \epsilon b_n \right\}}) - t (c_n -
    b_n c_\alpha(\epsilon)) \right] \right| \le A_n^\omega(t) + B_{n,
    \epsilon}^\omega(t) + C_{n, \epsilon}^\omega(t)
\end{equation}
with
\[
A_n^\omega(t) = \left| \frac{1}{b_n} \left[ \sum_{j=0}^{\lfloor nt \rfloor - 1} \bE_{\nu^{\sigma^j \omega}}(\phi) - t c_n  \right] \right|,
\]
\[
B_{n, \epsilon}^\omega(t) = \left| \frac{1}{b_n} \left[ \sum_{j=0}^{\lfloor nt \rfloor - 1} \mathbb{E}_{\nu^{\sigma^j \omega}}(\phi \one_{\left\{|\phi| > \epsilon b_n\right\}}) - nt \mathbb{E}_\nu(\phi \one_{\left\{| \phi | > \epsilon b_n\right\}})\right] \right|
\]
and
\[
C_{n, \epsilon}^\omega(t) = \left| \frac{nt}{b_n}\mathbb{E}_\nu(\phi \one_{\left\{| \phi | > \epsilon b_n\right\}}) - t c_\alpha(\epsilon) \right|.
 \]

 Since $\phi$ is regularity varying with index $\alpha >1$, it is
 integrable and the function $\omega \mapsto \bE_{\nu^\omega}(\phi)$ is
 H\"{o}lder. Hence, it satisfies the law of the iterated logarithm, and we
 have for $\bP$-a.e.~$\omega \in \Omega$
\[
\left|\frac{1}{k} \sum_{j=0}^{k-1} \bE_{\nu^{\sigma^j \omega}}(\phi) - \bE_\nu(\phi) \right| = \cO\left( \frac{\sqrt{\log \log k}}{\sqrt k}\right).
\]
Thus, we have
\[
\sup_{0 \le t \le \ell} A_n^\omega(t) = \mathcal{O} \left( \frac{\sqrt{n \ell} \sqrt{\log \log (n \ell)}}{b_n} \right).
\]
As a consequence, we can deduce that $\lim_{n \to \infty} \sup_{0 \le t \le \ell} A_n^\omega(t) = 0 $ since $b_n = n^{\frac{1}{\alpha}} \widetilde{L}(n)$ for a slowly varying function $\widetilde{L}$, with $\alpha <2$.

By Proposition \ref{prop:karamata}, we also have
\[
\lim_{n \to \infty} n b_n^{-1} \bE_\nu(\phi \one_{\left\{ |\phi| > \epsilon b_n \right\}}) = c_\alpha(\epsilon).
\]
In particular, we have \[
\lim_{n \to \infty} \sup_{0 \le t \le \ell} C_{n, \epsilon}^\omega(t) = 0. 
\]

This also implies that $\bE_m(| \chi_n| ) = \cO(n^{-1})$ if we define
$\chi_n = b_n^{-1}\phi \one_{\left\{ |\phi| > \epsilon b_n
  \right\}}$. From Lemma \ref{lemma:tech}, it follows that
$\lim_{n \to \infty} \sup_{0 \le t \le \ell} B_{n, \epsilon}^\omega(t) =
0$.

Putting all these estimates together concludes the proof when $\alpha \in (1,2)$.

\medskip

When $\alpha = 1$, we estimate the RHS of \eqref{eqn:iid_main} by
$A_{n, \epsilon}^\omega(t) + B_{n, \epsilon}^\omega(t)$ with
\[
A_{n, \epsilon}^\omega(t) = \left| \frac{1}{b_n} \left[ \sum_{j=0}^{\lfloor nt \rfloor - 1} \mathbb{E}_{\nu^{\sigma^j \omega}}(\phi \one_{\left\{|\phi| \le \epsilon b_n\right\}}) - nt\mathbb{E}_\nu(\phi \one_{\left\{| \phi | \le \epsilon b_n\right\}}) \right] \right|
\]
and
\[
B_{n, \epsilon}^\omega(t) = \left| \frac{nt}{b_n}\mathbb{E}_\nu(\phi \one_{\left\{\epsilon b_n <| \phi | \le b_n\right\}}) - t c_\alpha(\epsilon) \right|.
 \]

We define
$\chi_n = b_n^{-1} \phi \one_{\left\{| \phi | \le \epsilon
    b_n\right\}}$. By Proposition \ref{prop:karamata}, we have
$\bE_m(| \chi_n |) = \cO(n^{-1} \widetilde{L}(n))$ for some slowly varying
function $\widetilde{L}$, and so by Lemma~\ref{lemma:tech},
\[
\lim_{n \to \infty} \sup_{0 \le t \le \ell} A_{n, \epsilon}^\omega(t) = 0.
\]

On the other hand, by Proposition \ref{prop:karamata}, we have
\[
\lim_{n \to \infty} n b_n^{-1} \bE_\nu(\phi \one_{\left\{\epsilon b_n <|\phi| \le \epsilon b_n\right\}}) = c_\alpha(\epsilon)
\] 
and so $\lim_{n \to \infty} \sup_{0 \le t \le \ell} B_{n, \epsilon}^\omega(t) = 0$ which completes the proof.
\end{proof}


\section{An exponential law} \label{sec:exp_law}

We denote by $\cJ$ the family of all finite unions of intervals of the form $(x, y]$, where $- \infty \le x < y \le \infty$ and $0 \notin  [x,y]$. For $J \in \cJ$, we will establish a quenched exponential law for the sequence of sets $A_n = (\phi_{x_0})^{-1}(b_n J)$. Similar results were obtained in \cite{CS20, FFV17, HRY20, RSV14, RT15}.

 Since $\phi$ is regularly varying, it is easy to verify that 
\[
  \lim_{n \to \infty} n \nu(A_n) = \Pi_\alpha(J).
\]
In particular, $ m(A_n) = \cO(n^{-1})$.

\begin{lem}\label{lem:tech}
  Assume Condition $U$ and that $\phi$ is regularly varying with index
  $\alpha$. 

  If $A_n \subset [0,1]$ is a sequence of measurable subsets such that
  $m(A_n) = \cO(n^{-1})$, then for all $0 \le s < t$,
\[
\lim_{n \to \infty} \left(\left[ \sum_{j=\lfloor ns \rfloor +1}^{\lfloor nt \rfloor} \nu^{\sigma^j \omega}(A_n) \right] - n(t-s) \nu(A_n) \right) = 0.
\]
The same result holds in the setting of intermittent maps  if $A_n \subset [\delta, 1]$ for some $\delta >0$ with $m(A_n) = \cO(n^{-1})$.
In particular, if $A_n = \phi_{x_0}^{-1}(b_n J)$ for $J \in \cJ$, then for all $0 \le s < t$.
\[
\lim_{n \to \infty}  \sum_{j=\lfloor ns \rfloor+1}^{\lfloor nt \rfloor} \nu^{\sigma^j \omega}(A_n)  = (t-s) \Pi_\alpha(J).
\]
\end{lem}

\begin{proof} For the first statement, it suffices to apply Lemma~\ref{lemma:tech} or Lemma~\ref{lemma:tech_intermittent} with $\chi_n = \one_{A_n}$. The second statement immediately follows since $\lim_n n \nu(A_n) = \Pi_\alpha(J)$.
\end{proof}

\begin{cor}\label{cor:tech}
  Assume the hypothesis of Lemma~\ref{lem:tech}.

  Let $J \in \cJ$, and set $A_n = \phi^{-1}(b_n J)$. Then for $\bP$-a.e.
  $ \omega \in \Omega$, and all $0 \le s < t$,
\[
\lim_{n \to \infty} \prod_{j=\lfloor ns \rfloor +1}^{\lfloor nt \rfloor} \left( 1 - \nu^{\sigma^j \omega}(A_n)\right) = e^{-(t-s) \Pi_\alpha(J)}.
\]
\end{cor}

\begin{proof}
Since $\nu^{\omega}(A_n)$ is of order at most $n^{-1}$ uniformly in $\omega \in \Omega$, it follows that
\[
\log \left[\prod_{j=\lfloor ns \rfloor +1}^{\lfloor nt \rfloor} \left( 1 - \nu^{\sigma^j \omega}(A_n)\right) \right] = - \left( \sum_{j=\lfloor ns \rfloor +1}^{\lfloor nt \rfloor} \nu^{\sigma^j \omega}(A_n) \right) + \cO(n^{-1}).
\]
By Lemma \ref{lem:tech},
\[
\lim_{n \to \infty}  \sum_{j=\lfloor ns \rfloor}^{\lfloor nt \rfloor - 1} \nu^{\sigma^j \omega}(A_n)  = (t-s) \Pi_\alpha(J),
\]
which yields the conclusion.
\end{proof}

\begin{defin} \label{def:hitting_times} For a measurable subset $U \subset Y = [0,1]$, we define the
  hitting time of $(\omega, x) \in \Omega \times Y$ to $U$ by
\[
R_U(\omega)(x) := \inf \left\{ k \ge 1 \, : \, T_\omega^k(x) \in U \right\}.
\]
and the induced measure by $\nu$ on $U$  by
\[
\nu_{U} (A):=\frac{\nu (A\cap U)}{\nu (U)}
\]
\end{defin}

In order to establish our exponential law, we will first obtain a few estimates, based on the proof of \cite[Theorem 2.1]{HSV99}, to relate $\nu^\omega(R_{A_n}(\omega) > \lfloor nt \rfloor)$ to $\sum_{j=0}^{\lfloor nt \rfloor - 1} \nu^{\sigma^j \omega}(A_n)$ so that we are able to invoke Corollary \ref{cor:tech}.

The next lemma is basically \cite[Lemma 6]{RSV14}.

\begin{lem} \label{lem:return} For every measurable set $U \subset [0,1]$,
  we have the bound
  \begin{align*}\label{eq.return-time-bound}
    \left| \nu^{\omega} (R_U(\omega) > k) - \prod_{j=1}^{k} (1-\nu^{\sigma^j
    \omega}(U)) \right|&
  \le \sum_{j=1}^k \nu^{\sigma^j\omega} (U) \; c_{\sigma^j \omega}(k-j, U)
  \prod_{i=1}^{j-1} (1-\nu^{\sigma^i \omega}(U)) \\
&  \le \sum_{j=1}^k \nu^{\sigma^j\omega} (U) \; c_{\sigma^j \omega}(U)
  \end{align*}
where 
\[
  c_\omega(k,U):= \left|\nu_{U}^{\omega} (R_U(\omega)> k) -\nu^{\omega}
  (R_U(\omega)>k) \right|
\]
and
\[
c_\omega(U) := \sup_{k \ge 0} c_\omega(k,U).
 \]
\end{lem}

\begin{proof}
Note that
$\{R_U(\omega) > k\}=[T_{\omega}^{1}]^{-1} (U^c \cap \{R_U(\sigma \omega) >
k-1\})$ and so, using the equivariance of $\{\nu^\omega\}_{\omega \in \Omega}$,
\[
  \nu^{\omega} (R_U(\omega) > k) =\nu^{\sigma\omega} (U^c \cap
  \{R_U(\sigma\omega)> k-1\} ).
\]
Hence 
\[
\nu^{\omega} (R_U(\omega) > k)=\nu^{\sigma \omega} (R_U(\sigma \omega) > k-1) -\nu^{\sigma \omega} (U \cap \{R_U(\sigma \omega) > k-1\} ).
\]

We note that
\begin{align*}
  \nu^{\omega} (R_U(\omega) > k) &=\nu^{\sigma \omega} (R_U(\sigma \omega) > k-1)
  -\nu^{\sigma \omega} (U) [\nu^{\sigma \omega} (R_U(\sigma \omega) > k-1
  )+ c_{\sigma\omega}(k-1,U)] \\
& = (1 - \nu^{\sigma \omega}(U)) \nu^{\sigma \omega}(R_U(\sigma \omega) > k-1) - \nu^{\sigma \omega}(U) c_{\sigma \omega}(k-1, U).
\end{align*}

Iterating we obtain, using the fact that for $\bP$-a.e.~$\omega$, $\nu^{\omega} (R_U(\omega)\ge1)=1$,
\[
  \nu^{\omega} (R_U(\omega) > k)= \prod_{j=1}^{k} (1-\nu^{\sigma^j
    \omega}(U)) - \sum_{j=1}^k \nu^{\sigma^j\omega} (U) c_{\sigma^j
    \omega}(k-j, U) \prod_{i=1}^{j-1} (1-\nu^{\sigma^i \omega}(U))
\]
which yields the conclusion.
\end{proof}

We will estimate now the coefficients $c_\omega(U)$.

\begin{lem}\label{th.return-time-error-estimate}
  For any measurable subset $U \subset Y$ such that $\one_U \in {\rm BV}$, we have, for all $N$
  \begin{equation}\label{eq.c-U}
    c_\omega(U) \le \nu_U^{\omega} (R_U(\omega) \le N) +
    \nu^{\omega} (R_U(\omega) \le N)+
     \frac{1}{\nu^\omega(U)} \left\|  P_{\omega}^N ([\one_{U} -\nu^{\omega} (U)] h_{\omega})\right\|_{L^1(m)}
%
  \end{equation}
  and
  \begin{equation}\label{eq.nu-R-N}
    \nu_U^{\omega} (R_U(\omega) \le N) \le \frac{1}{\nu^\omega(U)}
    \nu^{\omega} (R_U(\omega) \le N),
    \qquad \nu^{\omega} (R_U(\omega) \le N) \le \sum_{i=1}^N \nu^{\sigma^i
      \omega}(U)
  \end{equation}

\end{lem}

\begin{proof}
  The estimates~\eqref{eq.nu-R-N} follow from
  \[
    \{R_U(\omega) \le N\} = \bigcup_{i=1}^N (T^i_\omega)^{-1}(U).
  \]
  and therefore
  \[
    \nu^{\omega} (R_U(\omega) \le N) \le
    \sum_{i=1}^N \nu^{\sigma^i \omega}(U)
  \]

  For~\eqref{eq.c-U}, note that
  \[
    c_\omega(U)=\left|\nu_U^{\omega} (R_U(\omega) \le j)-\nu^{\omega}
    (R_U(\omega) \le j)\right|
  \]

If $j \le N$ then
\[
  c_\omega(U) \le \nu_U^{\omega} (R(\omega) \le N) + \nu^{\omega}
  (R(\omega) \le N)
\]

If $j > N$ we write
\begin{align*}
  \nu_U^{\omega} (R_U(\omega) \le j)&-\nu^{\omega} (R_U(\omega) \le j)
  =~\nu_U^{\omega} (R_U(\omega) \le j) - \nu_{U}^{\omega}(T_{\omega}^{-N}
  (R_U(\sigma^N \omega) \le j-N)) \\
  &+~ \nu_{U}^{\omega}(T_{\omega}^{-N} (R_U(\sigma^N \omega) \le j-N)) -
  \nu^{\omega}(T_{\omega}^{-N} (R_U(\sigma^N \omega) \le j-N)) \\
  &+~\nu^{\omega}(T_{\omega}^{-N} (R_U(\sigma^N \omega) \le
  j-N))-\nu^{\omega} (R_U(\omega) \le j) \\
 &= (a) + (b) + (c).
\end{align*}

To bound $(a)$  and $(c)$ 
note that
\[
  \{R_U(\omega) \le j\} = \{R_U(\omega) \le N\} \cup T_{\omega}^{-N}
  (\{R_U(\sigma^N \omega) \le j-N)\})
\]
so 
\begin{equation}\label{eq.bound1}
  |\nu^{\omega} (R_U(\omega) \le j) - \nu^{\omega}(T_{\omega}^{-N}
  (R_U(\sigma^N \omega) \le j-N))| \le \nu^\omega(R_U(\omega) \le N)
\end{equation}
and similarly for $\nu_U^{\omega}$.

To bound $(b)$ we use the decay of $P_\omega^k$. Setting $V = \left\{ R_U(\sigma^N \omega) \le j - N \right\}$, we have

\begin{align*}
 |\nu_{U}^{\omega}(T_{\omega}^{-N} (V)) -
  \nu^{\omega}(T_{\omega}^{-N} (V))| 
  &=\frac{1}{\nu^\omega(U)}\left| \int_Y \one_{U} \one_{V}
    \circ T_{\omega}^N h_{\omega} dm
    - \nu^\omega(U) \int_Y \one_{V} \circ T_{\omega}^N
    h_{\omega} dm \right| \\
 & = \frac{1}{\nu^\omega(U)} \left| \int_Y \one_{V}
    P_{\omega}^N ([\one_{U} -\nu^{\omega} (U)] h_{\omega}) dm\right| \\
& \le \frac{1}{\nu^\omega(U)} \left\|  P_{\omega}^N ([\one_{U} -\nu^{\omega} (U)] h_{\omega})\right\|_{L^1(m)}.
\end{align*}
%
%
%
%
\end{proof}






%
%

\subsection{Exponential law: proof of Theorem~\ref{thm:return}} \label{ssec:exp_law_unif}

We can now prove the exponential law for $A_n = \phi^{-1}(b_n J)$,
$J \in \cJ$.

\begin{proof}[Proof of Theorem~\ref{thm:return}]

Due to rounding errors when taking the integer parts, we have
\begin{multline*}
\left| \nu^{\sigma^{\lfloor ns \rfloor}\omega} \left(R_{A_n}(\sigma^{\lfloor ns \rfloor}\omega) > \lfloor n(t-s) \rfloor\right) - \nu^{\sigma^{\lfloor ns \rfloor}\omega} \left(R_{A_n}(\sigma^{\lfloor ns \rfloor}\omega) > \lfloor nt \rfloor - \lfloor ns \rfloor \right)\right| \\
\le \nu^{\sigma^{\lfloor nt \rfloor}\omega}(A_n) \le C m(A_n) \to 0,
\end{multline*}
and it is thus enough to prove the convergence of $\nu^{\sigma^{\lfloor ns \rfloor}\omega} \left(R_{A_n}(\sigma^{\lfloor ns \rfloor}\omega) > \lfloor nt \rfloor - \lfloor ns \rfloor \right)$.

By Lemmas \ref{lem:return} and \ref{th.return-time-error-estimate}, for all
$N \ge 1$, we have
\begin{equation} \label{eqn:exp_law}
\left| \nu^{\sigma^{\lfloor ns \rfloor}\omega} \left(R_{A_n}(\sigma^{\lfloor ns \rfloor}\omega) > \lfloor nt \rfloor - \lfloor ns \rfloor \right) - \prod_{j = \lfloor ns \rfloor +1}^{\lfloor nt \rfloor} (1 - \nu^{\sigma^j \omega}(A_n) )\right| \le {\rm (I)} + {\rm (II)} + {\rm (III)},
\end{equation}
with
\[
{\rm (I)} = \sum_{j = \lfloor ns \rfloor +1}^{\lfloor nt \rfloor}\nu^{\sigma^j \omega}\left(A_n \cap \left\{R_{A_n}(\sigma^j \omega) \le N \right\}\right),
\]
\[
{\rm (II)} = \sum_{j = \lfloor ns \rfloor +1}^{\lfloor nt \rfloor} \nu^{\sigma^j \omega}(A_n) \nu^{\sigma^j \omega}(R_{A_n}(\sigma^j \omega) \le N)
\]
and
\[
{\rm (III)} = \sum_{j = \lfloor ns \rfloor +1}^{\lfloor nt \rfloor} \left\| P_{\sigma^j \omega}^N \left( \left[ \one_{A_n} - \nu^{\sigma^j \omega}(A_n)\right] h_{\sigma^j \omega} \right) \right\|_{L^1(m)}.
\]

To estimate (I), we choose $\epsilon > 0$ such that $J \subset \left\{ |x| > \epsilon \right\}$ and we introduce $V_n = \left\{ | \phi | > \epsilon b_n \right\}$. For a measurable subset $V \subset Y$, we also define the shortest return to $V$ by
\[
r_\omega(V) = \inf_{x \in V} R_V(\omega)(x),
\]
and we set
\[
r(V) = \inf_{\omega \in \Omega} r_\omega(V).
\]

We have
\begin{align*}
\nu^{\sigma^j \omega}\left(A_n \cap \left\{R_{A_n}(\sigma^j \omega) \le N \right\}\right) & \le \nu^{\sigma^j \omega}\left(V_n \cap \left\{R_{V_n}(\sigma^j \omega) \le N\right\}\right) \\
& \le \sum_{i=r_{\sigma^j \omega}(V_n)}^N \nu^{\sigma^j \omega}\left(V_n \cap  (T_{\sigma^j \omega}^i)^{-1}(V_n) \right) \\
& \le \sum_{i=r_{\sigma^j \omega}(V_n)}^N \int_Y \one_{V_n} P_{\sigma^j \omega}^i(\one_{V_n} h_{\sigma^j \omega}) dm.
\end{align*}

It follows from {\bf (Dec)} that
\begin{align*}
\left|\int_Y \one_{V_n} P_{\sigma^j \omega}^i(\one_{V_n} h_{\sigma^j \omega}) dm - \nu^{\sigma^j \omega}(V_n) \nu^{\sigma^{i+j} \omega}(V_n) \right| & \le \left\| \one_{V_n} \right\|_{L^1_m} \left\| P_{\sigma^j \omega}^i \left(\left[\one_{V_n} - \nu^{\sigma^j\omega}(V_n) \right] h_{\sigma^j \omega} \right)\right\|_{L^\infty_m} \\
& \le C \theta^i m(V_n) \left\| \left[\one_{V_n} - \nu^{\sigma^j\omega}(V_n) \right] h_{\sigma^j \omega} \right\|_{\rm BV} \\
& \le C \theta^i m(V_n),
\end{align*}
as BV is a Banach algebra, and both $\| \one_{V_n} \|_{\rm BV}$ and $\|h_{\sigma^j \omega} \|_{\rm BV}$ are uniformly bounded. \footnote{Recall that, from the definition of $\phi$, it follows that $V_n$ is an open interval, and thus $\one_{V_n}$ has a uniformly bounded BV norm.}.

Consequently, 
\begin{align*}
{\rm (I)} & \le \sum_{j = \lfloor ns \rfloor +1}^{\lfloor nt \rfloor} \sum_{i=r_{\sigma^j \omega}(V_n)}^N \left[ \nu^{\sigma^j \omega}(V_n) \nu^{\sigma^{i+j} \omega}(V_n) + \cO\left( \theta^i m(V_n)\right)\right] \\
& \le C \left(m(V_n)^2 n N + m(V_n) n \theta^{r(V_n)} \right).
\end{align*}

On the other hand, we have by \eqref{eq.nu-R-N},
\begin{align*}
{\rm (II)} & \le  \sum_{j = \lfloor ns \rfloor +1}^{\lfloor nt \rfloor} \nu^{\sigma^j \omega}(A_n) \sum_{i=1}^N \nu^{\sigma^{i+j} \omega}(A_n) \\
& \le C n N m(A_n)^2,
\end{align*}
and it follows from {\bf (Dec)} that
\begin{align*}
{\rm (III)} & \le C \theta^N  \sum_{j = \lfloor ns \rfloor +1}^{\lfloor nt \rfloor} \left\| \left[ \one_{A_n} - \nu^{\sigma^j \omega}(A_n)\right] h_{\sigma^j \omega} \right\|_{\rm BV} \\
& \le C n \theta^N,
\end{align*}
since $\{h_\omega\}_{\omega \in \Omega}$ is a bounded family in BV, $A_n$ is the union of at most two intervals and thus $\| \one_{A_n }\|_{\rm BV}$ is uniformly bounded. 
We can thus bound \eqref{eqn:exp_law} by 
\[
C \left( m(V_n)^2 n N + m(V_n) n \theta^{r(V_n)} + m(A_n)^2 n N + n \theta^N \right) \le C \left( n^{-1} N + \theta^{r(V_n)} + n \theta^N \right),
\]
and, assuming for the moment that $r(V_n) \to + \infty$, we obtain the conclusion by choosing $N = N(n) = 2 \log n$ and letting $n \to \infty$.

It thus remains to show that $r(V_n) \to + \infty$. Recall that $V_n$ is the ball of centre $x_0$ and radius $b^{-1} \epsilon^{-\alpha} n^{-1}$. Let $R \ge 1$ be a positive integer. Since $x_0$ is assumed to be non-recurrent, and that the collection of maps $T_\omega^j$ for $\omega \in \Omega$ and $0 \le j < R$ is finite, we have that
\[
\delta_R := \inf_{\omega \in \Omega}  \inf_{0 \le j < R} |T_\omega^j(x_0) - x_0| >0
\]
is positive. Since all the maps $T_\omega^j$ are continuous at $x_0$ by assumption, there exists $n_R \ge 1$ such that for all $n \ge n_R$, $j < R$ and $\omega \in \Omega$,
\[
x \in V_n \Longrightarrow | T_\omega^j(x) - T_\omega^j(x_0)| < \frac{\delta_R}{2}.
\]
Increasing $n_R$ if necessary, we can assume that $b^{-1} \epsilon^{-\alpha} n^{-1} < \frac{\delta_R}{2}$ for all $n \ge n_R$.

Then, for all $n \ge n_R$, $\omega \in \Omega$, $j < R$ and $x \in V_n$, we have
\[
|T_\omega^j(x) - x_0| \ge |T_\omega^j(x_0) - x_0| - |T_\omega^j(x) - T_\omega^j(x_0)| > \frac{\delta_R}{2} > b^{-1} \epsilon^{-\alpha} n^{-1},
\]
and thus $T_\omega^j(x) \notin V_n$.

This implies that $r(V_n) > R$ for all $n \ge n_R$, which concludes the
proof as $R$ is arbitrary. \end{proof}


\begin{rem} A quenched exponential
  law for random piecewise expanding maps of the interval is proved
  in Theorem 7.1~\cite[Section 7.1]{HRY20}. Our proof follows the same standard approach. We
  are able to specify that Theorem~\ref{thm:return} holds for non-recurrent
  $x_0$, since our assumptions imply decay of correlations against $L^1$
  observables, which is known to be necessary for this purpose, see
  \cite[Section 3.1]{AFV15}. Our proof is shorter, as we consider the
  simpler setting of finitely many maps, which are all uniformly expanding. In addition we
  use the exponential law in the intermittent case of Theorem 7.2~\cite[Section 7.2]{HRY20} to
  establish the short returns condition of Lemma~\ref{lem:return_intermittent_generic} below.
\end{rem} 

\subsection{Exponential law: proof of Theorem~\ref{thm:return_intermittent}}

In order to prove the exponential law in the intermittent setting,
Theorem~\ref{thm:return_intermittent}, we need a genericity condition on
the point $x_0$ in the definition \eqref{eq:phi-star} of $\phistar$.


\begin{lem} \label{lem:return_intermittent_generic}
  If
  $\gamma_{max}<\frac{1}{3}$, for $m$-a.e.~$x_0$ and
  for $\bP$-a.e.~$\omega\in \Omega$
  \[
    \lim_{n \to \infty} \sum_{j=\lfloor sn\rfloor +1}^{\lfloor tn \rfloor} m \left(B_{c n^{-1}}(x_0) \cap \left\{ R_{B_{c n^{-1}}(x_0)}^{\sigma^j \omega}  \le \lfloor n (\log n)^{-1} \rfloor \right\} \right) = 0.
  \]
 for all $c >0$ and all $0 \le s < t$.
\end{lem}
 
\begin{proof}
  Let $N = \lfloor n (\log n)^{-1} \rfloor$ an $V_n = B_{c n^{-1}}(x_0)$. First, we remark that for $m$-a.e.~$x_0$ and $\bP$-a.e.~$\omega$,
\begin{equation} \label{eqn:short_return_inter}
m \left(V_n \cap \left\{R_{V_n}(\omega) \le N \right\} \right) = o(n^{-1}).
\end{equation}
This is a consequence of~\cite[Theorem 7.2]{HRY20}. Their
  result is stated for two intermittent LSV maps both with
  $\gamma <\frac{1}{3}$ but generalizes immediately to a finite collection
  of maps with a uniform bound of $\gamma_{max}<\frac{1}{3}$. The
  exponential law for return times to nested balls imples that for a fixed $t$, for $m$-a.e $x_0$ and $\bP$-a.e.~$\omega$
  \[
\lim_{n \to \infty} \frac{1}{\nu^\omega(V_n)} \nu^\omega \left(V_n \cap \left\{R_{V_n}(\omega) \le nt \right\} \right) = 1-e^{-t}.
\]
which shows in particular, since   $\left\{R_{V_n}(\omega) \le N \right\} \subset  \left\{R_{V_n}(\omega) \le nt \right\}$ for all $n$ large enough, that for all $t>0$, $m$-a.e $x_0$ and $\bP$-a.e.~$\omega$
\begin{equation} \label{eqn:limsup_returntimes}
\limsup_{n \to \infty} \frac{1}{\nu^\omega(V_n)} \nu^\omega \left(V_n \cap \left\{R_{V_n}(\omega) \le N \right\} \right)  \le 1 - e^{-t}.
\end{equation}

Using \eqref{eq:intermittent-densities-comparable}, taking the limit $t \to 0$ proves \eqref{eqn:short_return_inter}. Note that, even though the set of full measure of $x_0$ and $\omega$ such that \eqref{eqn:limsup_returntimes} holds may depend on $t$, it is enough to consider only a sequence $t_k \to 0$. 

Now, for $k \ge 0$ and $n_0 \ge 1$, we introduce the set
\[
\Omega_k^{n_0} = \left\{ \omega \in \Omega \, : \, m\left(V_n \cap \left\{R_{V_n}(\omega) \le N \right\} \right) \le \frac{2^{-k}}{n} \text{ for all } n \ge n_0 \right\}.
\]

According to \eqref{eqn:short_return_inter}, we have for all $k \ge 0$,
\[
\lim_{n_0 \to \infty} \bP(\Omega_k^{n_0}) = \bP\left( \bigcup_{n_0 \ge 1} \Omega_k^{n_0} \right) = 1.
\]

By the Birkhoff ergodic theorem, for al $k \ge 0$, $n_0 \ge 1$ and $\bP$-a.e.~$\omega$, 
\[
\lim_{n \to \infty} \frac{1}{n} \sum_{j=0}^{n-1} \mathds{1}_{\Omega_k^{n_0}}(\sigma^j \omega) = \bP(\Omega_k^{n_0}),
\]
which implies that for all $0 \le s < t$,
\[
\lim_{n \to \infty} \frac{1}{\left(\lfloor nt \rfloor - \lfloor ns \rfloor\right)} \sum_{j=\lfloor ns \rfloor +1}^{\lfloor nt \rfloor} \mathds{1}_{\Omega_k^{n_0}}(\sigma^j \omega) = \bP(\Omega_k^{n_0}).
\]

Let $n_0 = n_0(\omega, k)$ such that $\bP(\Omega_k^{n_0}) \ge 1 - 2^{-k}$, and for all $n \ge n_0$,
\[
\frac{1}{\left(\lfloor nt \rfloor - \lfloor ns \rfloor \right)} \sum_{j=\lfloor ns \rfloor +1}^{\lfloor nt \rfloor} \mathds{1}_{\Omega_k^{n_0}}(\sigma^j \omega) \ge \bP(\Omega_k^{n_0}) - 2^{-k}.
\]

Then, for all $n \ge n_0(\omega, k)$ we have
\[
\frac{1}{\left( \lfloor nt \rfloor - \lfloor ns \rfloor \right)} \sum_{j=\lfloor ns \rfloor +1}^{\lfloor nt \rfloor} \mathds{1}_{(\Omega_k^{n_0})^c}(\sigma^j \omega) \le 2^{-(k-1)}.
\]

Consequently,
\begin{align*}
\sum_{\lfloor ns \rfloor +1}^{\lfloor nt \rfloor} m\left(V_n \cap \left\{R_{V_n}(\omega) \le N \right\} \right)  \le \left(\lfloor nt \rfloor - \lfloor ns \rfloor \right) \frac{2^{-k}}{n} +   \left(\lfloor nt \rfloor - \lfloor ns \rfloor \right) 2^{-(k-1)} m(V_n) .
\end{align*}
This proves that 
\[
\limsup_{n \to \infty} \sum_{\lfloor ns \rfloor +1}^{\lfloor nt \rfloor} m\left(V_n \cap \left\{R_{V_n}(\omega) \le N \right\} \right) \le C \, 2^{-k},
\]
and the result follows by taking the limit $k \to \infty$.

Note that the set of $x_0$ and $\omega$ for which the lemma holds depends a priori on $c>0$, but it is enough to consider a countable and dense set of $c$, since for $c < c'$,
\[
\left\{ B_{c n^{-1}}(x_0) \cap \left\{ R_{B_{c n^{-1}}(x_0)}^\omega \le N \right\} \right\} \subset \left\{ B_{c' n^{-1}}(x_0) \cap \left\{ R_{B_{c' n^{-1}}(x_0)}^\omega \le N \right\} \right\}.
\]
\end{proof}

%
%

The exponential law for random intermittent maps follows from Lemma~\ref{lem:return_intermittent_generic}:

\begin{proof}[Proof of Theorem~\ref{thm:return_intermittent}]
  We consider the three terms in~\eqref{eqn:exp_law} with $N = \lfloor n (\log n)^{-1} \rfloor$.

  Let $V_n = \left\{ | \phi | > \epsilon b_n\right\}$ where $\epsilon >0$ is such that $A_n \subset V_n$ for all $n\ge 1$. Since $V_n$ is a ball of centre $x_0$ and radius $b^{-1} \epsilon^{-\alpha} n^{-1}$, and since $V_n \subset [\delta, 1]$,  the term 
  \[ {\rm (I)} = \sum_{j = \lfloor ns \rfloor +1}^{\lfloor nt
      \rfloor}\nu^{\sigma^j \omega}\left(A_n \cap \left\{R_{A_n}(\sigma^j
        \omega) \le N \right\}\right) \le C \sum_{j = \lfloor ns \rfloor +1}^{\lfloor nt
      \rfloor} m\left(V_n \cap \left\{R_{V_n}(\sigma^j
        \omega) \le N \right\}\right)
  \]
  tends to zero  by Lemma~\ref{lem:return_intermittent_generic} for $m$-a.e~ $x_0$.

The term 
\[
  {\rm (II)} = \sum_{j = \lfloor ns \rfloor +1}^{\lfloor nt \rfloor}
  \nu^{\sigma^j \omega}(A_n) \nu^{\sigma^j \omega}(R_{A_n}(\sigma^j \omega)
  \le N)
  \le CnNm(A_n)^2
\]
also tends to zero since $N=o(n)$. Lastly we consider 
\[
  {\rm (III)} = \sum_{j = \lfloor ns \rfloor +1}^{\lfloor nt \rfloor} \left\| P_{\sigma^j \omega}^N \left( \left[ \one_{A_n} - \nu^{\sigma^j \omega}(A_n)\right] h_{\sigma^j \omega} \right) \right\|_{L^1(m)}.
\]
    
We approximate $\one_{A_n}$ by a $C^1$ function $g$ such that
$\|g\|_{C^1}\le n^{\tau}$, $g=\one_{A_n}$ on $A_n$ and
$\|g-\one_{A_n}\|_{L^1}\le n^{-\tau}$ (recall $A_n$ is two intervals
of length roughly $\frac{1}{n}$ so a simple smoothing at the endpoints of
the intervals allows us to find such a function $g$). Later we will specify
$\tau >1$ will suffice. By~\cite[Lemma 3.4]{NPT} with $h = h_\omega$ and $\varphi = g - m(g h_\omega)$, for all $\omega$,
    \begin{align*}
      \left\| P_{\omega}^N ([g-m(g h_\omega)] h_{\omega})\right\|_{L^1}&\le C n^\tau N^{1-\frac{1}{\gamma_{max}}} (\log N)^{\frac{1}{\gamma_{max}}} \\
& \le C n^{\tau +1 - \frac{1}{\gamma_{max}}} (\log n)^{\frac{2}{\gamma_{max}} - 1}.
    \end{align*}
    Using the decomposition
    $\one_{A_n} -\nu^{\omega} (A_n)=(\one_{A_n}-g)
    -(\nu^{\omega} (A_n)-m(gh_\omega))+(g-m(gh_\omega))$ we estimate, leaving out the log term,
 
    \[
   {\rm (III)} \le C \left[n^{1-\tau}+n^{\tau
       +2-\frac{1}{\gamma_{max}}}\right]
    \]
    where the value of $C$ may change line to line. Taking
    $\gamma_{max}<\frac{1}{3}$ and $1<\tau <\frac{1}{\gamma_{max}}-2$
    suffices.
\end{proof}
 
\section{Point process results}

We now proceed to the proof of the Poisson convergence. In
Section~\ref{sec:annealed} we will consider an annealed version of our
results.



\subsection{Uniformly expanding maps: proof of Theorem~\ref{thm:poisson_expanding}}\label{ssec:cv_poisson}

Recall Theorem~\ref{thm:poisson_expanding}: under the conditions of
Section~\ref{sec:expanding_iid}, in particular {\bf (LY)}, {\bf (Min)} and {\bf (Dec)}, if
$x_0 \notin \cD$ is non-recurrent, then for $\bP$-a.e.
$\omega \in \Omega$
\[
  N_n^\omega \dto N_{(\alpha)}
\]
under the probability measure  $\nu^\omega$.

Our proof of Theorem~\ref{thm:poisson_expanding} uses the existence of a
spectral gap for the associated transfer operators $P^n_{\omega}$, and
breaks down in the setting of intermittent maps. The use of the
spectral gap is encapsulated in the following lemma.

\begin{lem}\label{th.bound-Pn}
  Assume {\bf (LY)}. Then there exists $C >0$ such that for all
  $\omega \in \Omega$, all $f, f_n \in {\rm BV}$ with
  \[
    \sup_{j \ge 1} \| f_j \|_{L^\infty(m)} \le 1 \: \text{ and } \: \sup_{j
      \ge 1} \|f_j \|_{\rm BV} < \infty,
  \]
  we have
  \[
    \sup_{n \ge 0} \left\| P_\omega^n \left( f \cdot \prod_{j=1}^n f_j
        \circ T_\omega^j \right)\right\|_{\rm BV} \le C \|f\|_{\rm BV}
    \left( \sup_{j \ge 1} \| f_j \|_{\rm BV} \right)
  \]
\end{lem}
  %
  %

\begin{proof} We proceed in four steps.

{\bf Step 1. } We define
\[
g_\omega^n = \prod_{j=0}^n f_j \circ T_\omega^j,
\]
where we have set $f_0 = \one$. We observe that for all $n \ge 0$, there exists $C_n > 0$ such that for all $\omega \in \Omega$,
\begin{equation}\label{eqn:g_n}
  \|g_\omega^n \|_{L^\infty(m)} \le \left(
    \sup_{j \ge 1} \| f_j \|_{L^\infty(m)} \right)^{n+1} \le 1 \: \text{
    and } \: \|g_\omega^n\|_{\rm BV} \le C_n \left( \sup_{j \ge 1}
    \|f_j\|_{\rm BV} \right).
\end{equation}
The first estimate is immediate, and the second follows, because
\begin{align*}
\var(g_\omega^{n+1}) &\le \var(g_\omega^n) \|f_{n+1} \circ T_\omega^{n+1}\|_{L^\infty(m)} + \|g_\omega^n\|_{L^\infty(m)} \var(f_{n+1}\circ T_\omega^{n+1}) \\
& \le  \var(g_\omega^n) + \var(f_{n+1}\circ T_\omega^{n+1}) \\
& = \var(g_\omega^n) + \sum_{I \in \cA_\omega^{n+1}} \var_I(f_{n+1}\circ T_\omega^{n+1}) \\
& = \var(g_\omega^n) + \sum_{I \in \cA_\omega^{n+1}} \var_{T_\omega^{n+1}(I)}(f_{n+1}) \\
& \le \var(g_\omega^n) + \left(\#  \cA_\omega^{n+1}\right) \var(f_{n+1}),
\end{align*}
and so we can define by induction $C_{n+1} = C_n + \sup_{\omega \in \Omega} \# \cA_\omega^{n+1}$ which is finite, as there are only finitely many maps in $\cS$.

{\bf Step 2. } We first prove the lemma in the case where $r=1$ in the condition {\bf (LY)}. Before, we claim that for $f \in {\rm BV}$ and sequences $(f_j) \subset {\rm BV}$ as in the statement, we have
\begin{multline} \label{eqn:claim}
\var\left(P_\omega^n(f g_\omega^n)\right) \le \sum_{j=0}^n \rho^j  \| P_\omega^{n-j}(f g_\omega^{n-j-1}) \|_{L^\infty(m)} \|f_{n-j}\|_{\rm BV} \\ + D \sum_{j=0}^{n-1} \rho^j \|P_\omega^{n-1-j}(f g_{\omega}^{n-1-j})\|_{L^1(m)} \|f_{n-j}\|_{L^\infty(m)}.
\end{multline}

This implies the lemma when $r=1$, since
\[
\| P_\omega^{n-j}(f g_\omega^{n-j-1}) \|_{L^\infty(m)} \le \|g_\omega^{n-j-1}\|_{L^\infty(m)} \|P_\omega^{n-j} |f| \|_{L^\infty(m)} \le C \|f\|_{\rm BV},
\]
and
\[
 \|P_\omega^{n-j}(f g_{\omega}^{n-j})\|_{L^1(m)} \le \|f g_{\omega}^{n-j} \|_{L^1(m)}\le  \|f\|_{L^\infty(m)} \|g_\omega^{n-j}\|_{L^1(m)} \le \|f\|_{\rm BV}.
\]

We prove the claim by induction on $n \ge 0$. It is immediate for $n=0$, and for the induction step, we have, using {\bf (LY)},
\begin{align*}
\var &(P_\omega^{n+1} (f g_\omega^{n+1})) \\ & = \var(P_\omega^{n+1}(f g_\omega^n f_{n+1} \circ T_\omega^{n+1})) = \var(P_\omega^{n+1}(f g_\omega^n) f_{n+1}) \\
& \le \var(P_\omega^{n+1}(f g_\omega^n)) \|f_{n+1}\|_{L^\infty(m)} + \|P_\omega^{n+1} (f g_\omega^n) \|_{L^\infty(m)} \var(f_{n+1}) \\
& \le \left( \rho \var(P_\omega^{n}(f g_\omega^n)) + D \|P_\omega^n (f g_\omega^n) \|_{L^1(m)} \right) \|f_{n+1}\|_{L^\infty(m)} + \|P_\omega^{n+1} (f g_\omega^n) \|_{L^\infty(m)} \var(f_{n+1}) \\
& \le \rho \var(P_\omega^{n}(f g_\omega^n)) + D \|P_\omega^n (f g_\omega^n) \|_{L^1(m)} \|f_{n+1}\|_{L^\infty(m)} + \|P_\omega^{n+1} (f g_\omega^n) \|_{L^\infty(m)}  \|f_{n+1}\|_{\rm BV},
\end{align*}
which proves \eqref{eqn:claim} for $n+1$, assuming it holds for $n$.

{\bf Step 3. } Now, we consider the general case $r \ge 1$ and we assume
that $n$ is of the particular form $n = pr$, with $p \ge 0$. We note that
the random system defined with
$\cT = \left\{ T_\omega^r \right\}_{\omega \in \Omega}$ satisfies the
condition {\bf (LY)} with $r = 1$. Consequently, by the second step and
\eqref{eqn:g_n}, we have
\begin{align*}
\left\| P_\omega^n (f g_\omega^n)\right\|_{\rm BV} & = \left\| P_{\sigma^{r-1} \omega}^r \circ \ldots \circ P_\omega^r \left(f \prod_{j=1}^p g_{\sigma^{jr}\omega}^r \circ T_\omega^{jr} \right) \right\|_{\rm BV} \\
& \le C \|f\|_{\rm BV} \left( \sup_{j \ge 1} \| g_{\sigma^{jr} \omega}^r
  \|_{\rm BV}\right) 
 \le C C_r \|f\|_{\rm BV} \left(\sup_{j \ge 1} \|f_j\|_{\rm BV} \right).
\end{align*}

{\bf Step 4. } Finally, if $n = pr + q$, with $p \ge 0$ and $q \in \{0, \ldots, r-1\}$,  as an immediate consequence of {\bf (LY)}, we obtain
\begin{align*}
 \|P_\omega^n(f g_\omega^n) \|_{\rm BV} & = \|P_{\sigma^{pr}\omega}^q P_\omega^{pr}(f g_\omega^{pr} g_{\sigma^{pr}\omega}^q \circ T_\omega^{pr})\|_{\rm BV} \\
& = \|P_{\sigma^{pr}\omega}^q (P_\omega^{pr}(f g_\omega^{pr})g_{\sigma^{pr} \omega}^q) \|_{\rm BV}
 \le C \| P_\omega^{pr}(f g_\omega^{pr})g_{\sigma^{pr} \omega}^q\|_{\rm BV}.
\end{align*}

But, from Step 3, we have
\begin{align*}
\|  P_\omega^{pr}(f g_\omega^{pr})g_{\sigma^{pr} \omega}^q \|_{L^1(m)} & \le \| g_{\sigma^{pr} \omega}^q \|_{L^\infty(m)} \| P_\omega^{pr} (f g_\omega^{pr})\|_{L^1(m)} \\
& \le  \| P_\omega^{pr} (f g_\omega^{pr})\|_{L^1(m)} 
 \le C \|f\|_{\rm BV} \left(\sup_{j \ge 1} \|f_j \|_{\rm BV}\right),
\end{align*}

and, using \eqref{eqn:g_n},

\begin{align*}
\var( P_\omega^{pr}(f g_\omega^{pr})g_{\sigma^{pr} \omega}^q) & \le \|P_\omega^{pr}(f g_\omega^{pr})\|_{L^\infty(m)} \var(g_{\sigma^{pr}\omega}^q) + \var(P_\omega^{pr}(f g_\omega^{pr})) \| g_{\sigma^{pr}\omega}^q \|_{L^\infty(m)} \\
& \le \left[ C_q \|g_\omega^{pr}\|_{L^\infty(m)} \|P_\omega^{pr}  |f|\|_{L^\infty(m)} + C \|f\|_{\rm BV}\right] \left( \sup_{j \ge 1 } \| f_j\|_{\rm BV}\right) \\
& \le C \left( 1 + \max_{q = 0, \ldots, r-1} C_q \right) \|f\|_{\rm BV} \left( \sup_{j \ge 1 } \| f_j\|_{\rm BV}\right),
\end{align*}
which concludes the proof of the lemma.
\end{proof}

\begin{proof}[Proof of Theorem~\ref{thm:poisson_expanding}]

We denote by $\cR$ the family of finite unions of rectangles $R$ of the form $R = (s,t] \times J$ with $J \in \cJ$. By Kallenberg's theorem, see \cite[Theorem 4.7]{Kal75} or \cite[Proposition 3.22]{Res87}, $N_n^\omega \dto N_{(\alpha)}$ if for any $R \in \cR$, 

\[(a)~\lim_{n \to \infty} \nu^\omega(N_n^\omega(R) = 0) = \bP(N_{(\alpha)}(R) = 0),\]
and
\[(b)~\lim_{n \to \infty} \bE_{\nu^\omega}N_n^\omega(R) = \bE N_{(\alpha)}(R).\]

We first prove (b). We write 
\[
R =  \bigcup_{i=1}^k R_i,
\]
with $R_i = (s_i, t_i] \times J_i$ disjoint.

Then
\[
\bE N_{(\alpha)}(R) = \sum_{i=1}^k (t_i - s_i) \Pi_\alpha(J_i)
\]
and
\begin{align*}
\bE_{\nu^\omega}N_n^\omega(R) = \sum_{i= 1}^k \bE_{\nu^\omega}N_n^\omega((s_i, t_i] \times J_i) &= \sum_{i=1}^k \sum_{n s_i < j \le n t_i} \bE_{\nu^\omega}(\one_{\phi_{x_0}^{-1}(b_n J_i)} \circ T_\omega^{j-1}) \\ &= \sum_{i=1}^k \sum_{n s_i < j \le n t_i} \nu^{\sigma^{j-1}\omega}(\phistar^{-1}(b_n J_i)) \\
&= \sum_{i=1}^k \sum_{j=\lfloor ns_i \rfloor}^{\lfloor nt_i \rfloor - 1} \nu^{\sigma^j \omega}(\phistar^{-1}(b_n J_i)).
\end{align*}

By Lemma \ref{lem:tech}, for $\bP$-a.e.~$\omega \in \Omega$, we have 
\[
\lim_{n \to \infty} \sum_{i=1}^k \sum_{j=\lfloor ns_i \rfloor}^{\lfloor nt_i \rfloor - 1} \nu^{\sigma^j \omega}(\phi_{x_0}^{-1}(b_n J_i)) = (t_i - s_i) \Pi_\alpha(J_i),
\]
which proves (b).

We next establish (a). We will use induction on the number of ``time'' intervals
  $(s_i,t_i] \subset (0,\infty]$. Let $R=(s_1,t_1]\times J_1$ where
  $J_1\in \mathcal{J}$. Define
  \[
    A_n =\phistar^{-1} (b_n J_1).
  \]
  Since 
  \begin{align*}
    \{N_n^{\omega} (R)=0\}
    & = \{ x: T_{\omega}^j (x)\not \in A_n, ns_1<j+1 \le nt_1 \} \\
&= \left\{ 1_{A_n^c}\circ T_{\omega}^{\lfloor ns_1 \rfloor} \cdot 1_{A_n^c}\circ
      T_{\omega}^{\lfloor ns_1 \rfloor+1} \cdot \ldots \cdot 1_{A_n^c}\circ
      T_{\omega}^{\lfloor nt_1 \rfloor-1} \not= 0\right\} \\
&    = \left\{ x: \left(\prod_{j=0}^{\lfloor n t_1 \rfloor  -1 - \lfloor n s_1 \rfloor}1_{A_n^c}\circ
        T_{\sigma^{\lfloor n s_1 \rfloor}\omega}^{j}\right) \circ T_{\omega}^{\lfloor n s_1 \rfloor}(x) \not= 0\right\},
  \end{align*}
  we have that, 
  \begin{multline} \label{eqn:rounding}
    \left|\nu^{\omega} (N_n^{\omega} (R)=0) - \nu^{\sigma^{\lfloor ns_1 \rfloor} \omega}
      \left(R_{A_n}(\sigma^{\lfloor ns_1 \rfloor} \omega) > \lfloor n(t_1-s_1) \rfloor \right)\right| \\
 \le \nu^{\sigma^{\lfloor n s_1 \rfloor} \omega}(R_{A_n}(\sigma^{\lfloor n s_1 \rfloor}\omega) = 0) 
    = \nu^{\sigma^{\lfloor n s_1 \rfloor}\omega}(A_n) \le C m(A_n)\to 0,
  \end{multline}
  because, due to rounding when taking integer parts,
  $\lfloor n t_1 \rfloor - \lfloor n s_1 \rfloor -1$ is either equal to
  $\lfloor n( t_1 - s_1) \rfloor -1$ or to $\lfloor n(t_1 - s_1) \rfloor$.
  By Theorem~\ref{thm:return},
  \[
    \nu^{\sigma^{\lfloor ns_1\rfloor} \omega} (R_{A_n}(\sigma^{\lfloor ns_1 \rfloor} \omega) >
    \lfloor n(t_1-s_1) \rfloor)\to e^{-(t_1-s_1) \Pi_{\alpha} (J)}
  \]
  as desired.

  Now let $R=\cup_{j=1}^k (s_i, t_i]\times J_i$ with
  $0\le s_1 < t_1<\ldots < s_k < t_k$ and $J_i \in \mathcal{J}$.
  Furthermore, define $s_i'=s_i-s_1$ and $t_i'=t_i-s_1$.

  Observe that, accounting for the rounding errors when taking integer parts as for \eqref{eqn:rounding}, we get
  \begin{multline}\label{eq.rounding}
    \left|\nu^{\omega} \left(N_n^{\omega} \left(\bigcup_{i=1}^k (s_i,t_i]\times
        J_i\right)=0\right) -\nu^{\sigma^{\lfloor n s_1 \rfloor}\omega} \left(N_n^{\sigma^{\lfloor n
            s_1 \rfloor}\omega} \left(\bigcup_{i=1}^k (s'_i,t'_i]\times
        J_i\right)=0\right)\right| \\ \le 2C \sum_{i=1}^k m (\phistar^{-1}(b_n J_i)) \to 0
  \end{multline}
  so, after replacing $\omega$ by $\sigma^{\lfloor n
    s_1 \rfloor}\omega$, we can assume that $s_1=0$. Let
\begin{align*}
  &R_1=(0,t_1] \times J_1 \\
&R_2=\bigcup_{i=2}^k (s_i,t_i]\times J_i\\
  &R'_2=\bigcup_{i=2}^k (s_i-s_2,t_i-s_2]\times J_i
\end{align*}


Then, with $A_n =\phistar^{-1} (b_n J_1)$,
\begin{equation} \label{eq.mixing2}
  \left|\nu^{\eta} \left(N_n^{\eta} \left(R_1\cup R_2\right)=0\right)-\nu^{\eta}
    \left[\left\{R_{A_n}(\eta)> \lfloor nt_1 \rfloor\right\}\cap T_{\eta}^{-\lfloor ns_2\rfloor}
      \left(N_n^{\sigma^{ \lfloor ns_2\rfloor}\eta}(R'_2)=0\right)\right]\right|\to 0
\end{equation}
as $n\to \infty$, uniformly in $\eta \in \Omega$, as in~\eqref{eq.rounding}. Moreover,
as we check below,
\begin{multline}\label{eq.mixing}
\Big\vert\nu^{\eta} \left[\{R_{A_n}(\eta)> \lfloor nt_1 \rfloor\} \cap T_{\eta}^{-\lfloor ns_2 \rfloor} \left(N_n^{\sigma^{\lfloor ns_2\rfloor}\eta}(R'_2)=0\right)\right] \\
- \nu^{\eta} (R_{A_n}(\eta)> \lfloor nt_1 \rfloor)\cdot \nu^{\eta} (N_n^{\eta}(R_2)=0) \Big\vert\to 0
\end{multline}
as $n\to\infty$, uniformly in $\eta \in \Omega$. Therefore, setting $\eta = \sigma^{\lfloor n s_2 \rfloor} \omega$ in \eqref{eq.mixing2} and \eqref{eq.mixing}, we have, by
Theorem~\ref{thm:return},
\[
  \lim_{n\to\infty} \left| \nu^{\sigma^{\lfloor n s_2 \rfloor} \omega} (N_n^{\sigma^{\lfloor n s_2 \rfloor} \omega}(R_1 \cup R_2)=0) - e^{-t_1
      \Pi_{\alpha} (J_1)}\nu^{\sigma^{\lfloor n s_2 \rfloor} \omega} (N_n^{\sigma^{\lfloor n s_2 \rfloor} \omega}(R_2)=0) \right| = 0
\]
which gives the induction step
in the proof of (a).

We prove now \eqref{eq.mixing}. Our proof uses the spectral gap for
$P^{n}_{\omega}$ and breaks down for random intermittent maps.


Similarly to~\eqref{eq.rounding},
\[
  \left| \nu^\eta(N_n^{\eta} (R_2)=0) - \nu^\eta(T_{\eta}^{-\lfloor ns_2 \rfloor}
    (N_n^{\sigma^{\lfloor ns_2 \rfloor}\eta}(R'_2)=0)) \right| \to 0 \text{ as
    $n\to\infty$, uniformly in $\eta$.}
\]
We have, using the notation
\[
  U=\left\{R_{A_n}(\eta)> \lfloor nt_1 \rfloor\right\}, \qquad
  V=\left\{N_n^{\sigma^{\lfloor ns_2 \rfloor}\eta}(R'_2)=0\right\},
\]
that
\[
  \begin{aligned}\nonumber
     \Big\vert \nu^{\eta} \left( U \cap T_{\eta}^{-\lfloor ns_2 \rfloor} (V)\right) & - \nu^{\eta} (U)
    \nu^{\eta}\left(T_{\eta}^{-\lfloor ns_2 \rfloor}(V)\right)\Big\vert\\
& = \left|\int P_\eta^{ \lfloor n
      s_2 \rfloor}\left((\one_U-\nu^\eta(U))h_\eta\right) \one_V d m\right|
    \\
    & \le C \left\| P_{\eta}^{\lfloor n s_2 \rfloor}\left((\one_U - \nu^\eta(U))h_\eta\right)\right\|_{BV} \\
    &= \left\| P_{\sigma^{\lfloor n t_1 \rfloor}\eta}^{\lfloor n s_2\rfloor - \lfloor n t_1 \rfloor}P_{\eta}^{\lfloor n t_1\rfloor}\left((\one_U
    - \nu^\eta(U))h_\eta\right)\right\|_{\rm BV}
    \\
    & \le C\theta^{\lfloor n s_2\rfloor-\lfloor n t_1\rfloor} \left\| P_{\eta}^{\lfloor n t_1 \rfloor}\left(( \one_{U} - \nu^\eta(U))h_\eta\right)\right\|_{\rm BV}
  \end{aligned}
\]
where the last inequality follows from the decay, uniform in $\eta$, of
$\{P_{\eta}^k\}_{k}$ in BV (condition {\bf (Dec)}).


But
\begin{equation}\label{eq.bound-Pn}
  \sup_\eta \sup_n \left\|P_{\eta}^{\lfloor nt_1 \rfloor} \left((\one_{\left\{R_{A_n}(\eta)> \lfloor nt_1 \rfloor\right\}} -
  \nu^\eta({R_{A_n}(\eta)> \lfloor nt_1 \rfloor}))h_\eta\right)
 \right \|_{\rm BV} < \infty,
\end{equation}
which proves~\eqref{eq.mixing}. This follows from Lemma~\ref{th.bound-Pn}
below applied to $f = h_\eta$ and $f_j = \one_{A_n^c}$, because
\begin{equation*}
  \one_{\left\{R_{A_n}(\eta)> \lfloor nt_1 \rfloor \right\}} =\prod_{j=1}^{\lfloor nt_1 \rfloor}\one_{A_n^c} \circ T_{\eta}^j,
\end{equation*}
and both $\|h_\eta\|_{\rm BV}$ and $\| \one_{A_n^c}\|_{\rm BV}$ are uniformly bounded.
Note that for the stationary case the estimate~\eqref{eq.bound-Pn} is used
in the proof of~\cite[Theorem 4.4]{TK-dynamical}, which refers
to~\cite[Proposition 4]{ADSZ04}. \end{proof}

 \subsection{Intermittent maps: proof of Theorem~\ref{intermittent_poisson}}

We prove a weaker form of convergence in the setting of intermittent maps, which 
suffices to establish stable limit laws but not functional limit laws. 

 In the setting of intermittent maps,  we will show that for  $\bP$-a.e.~$\omega$,
  \[
    N^{\omega}_n ((0,1]\times \cdot ) \dto N_{(\alpha)} ((0,1] \times \cdot
    )
  \]

\begin{proof}[Proof of Theorem~\ref{intermittent_poisson}]

  We will show that for $\bP$-a.e.~$\omega \in \Omega$,
  the assumptions of Kallenberg's theorem \cite[Theorem 4.7]{Kal75}
  hold.

    Recall that $\mathcal{J}$ denotes the set of all finite unions of
    intervals of the form $(x,y]$ where $x <y$ and $0\not \in [x,y]$.

    By Kallenberg's theorem \cite[Theorem 4.7]{Kal75},
    $N^{\omega}_n[(0,1]\times \cdot) \to^{d} N_{(\alpha)} ((0,1] \times
    \cdot )$ if for all $J \in \mathcal{J}$,
  \[
    (a) \lim_{n\to \infty} \nu^{\omega} (N_n^{\omega} ((0,1] \times J)=0)=\bP
    (N_{(\alpha)} ((0,1]\times J)=0)
  \]
  and
  \[
    (b) \lim_{n\to \infty} \bE_{\nu^{\omega}} N^{\omega}_n ((0,1]\times J)=
    \bE[N_{(\alpha)} ((0,1]\times J)]
  \]

  We prove first $(b)$ following \cite[page 12]{TK-dynamical}.
  Write
  \[
    J=\bigcup_{i=1}^k J_i
  \]
  with $J_i=(x_i,y_i]$ disjoint.
  
  Then
  \[
    \bE N_{(\alpha)} ((0,1]\times J)=\sum_{i=1}^k \Pi_{\alpha}
    (J_i)=\Pi_{\alpha} (J)
  \]
  and
  \[
    \bE_{\nu^{\omega}} N_n^{\omega} ((0,1]\times J)= \sum_{i=1}^k
    \sum_{j=1}^{n} \bE_{\nu^{\omega}} [ \one_{(\phistar^{-1} (b_n
      J_i))}\circ T_{\omega}^{j-1}] = \sum_{j=1}^{n} \bE_{\nu^{\omega}} [
    \one_{(\phistar^{-1} (b_n J))}\circ T_{\omega}^{j-1}]
  \]
  We check that
  \[
    \lim_{n\to \infty} \sum_{j=1}^{n} \bE_{\nu^{\omega}} \left(
      \one_{\{\phistar^{-1} (b_n J)\}}\circ T_{\omega}^j\right) =
    \Pi_{\alpha} (J)
  \]
  for $J=\cup_{i=1}^k J_i$.

  Write $A_n:=\phistar^{-1} (b_n J)$.
  Then
  \[
    \bE_{\nu^{\omega}} [\one_{(\phi_{x_0}^{-1} (b_n J))}\circ T_{\omega}^j ] =\nu^{\sigma^j\omega}(A_n)
  \]
  hence
  \[
    \lim_{n\to \infty} \sum_{j=1}^{n} \bE_{\nu^{\omega}} [ \one_{(\phistar^{-1} (b_n
      J_i))}\circ T_{\omega}^j (x)] = \Pi_{\alpha} (J)
  \]
  by Lemma~\ref{lemma:tech_intermittent}.
  
  Now we prove (a), i.e. 
  \[
    \lim_{n\to \infty} \nu^{\omega} (N_n^{\omega} ((0,1] \times
    J)=0)=P(N_{(\alpha)} ((0,1]\times J)=0)
  \]
  for all $J \in \mathcal{J}$.

  Let $J \in \mathcal{J}$ and denote as above
  $A_n := \phistar^{-1} (b_n J)\subset X =[0,1]$. Then
  \[
    \{N_{n}^{\omega} ((0,1]\times J)=0\}=\{ x: T_{\omega}^{j} (x) \not \in
    A_n, 0 < j+1 \le n\} = \{R_{A_n}(\omega) > n-1\}\cap A^c_n
  \]
  Hence
  \[
    | \nu^{\omega} (N_n^\omega ((0,1]\times J)=0) -\nu^{\omega}
    (R_{A_n}(\omega) > n)|    \le C m(A_n) \to 0
  \]
  and by Theorem~\ref{thm:return_intermittent}, for $m$-a.e.~$x_0$
  \[
    \nu^{ \omega} (R_{A_n}(\omega) > n) \to e^{- \Pi_{\alpha} (J)}.
  \]
  This proves (a).
\end{proof}


\section{Stable laws and functional limit laws}\label{ssec:proof}

\subsection{Uniformly expanding maps: proof of Theorem~\ref{thm:expanding}} \label{ssec:proof_unif}

In this section, we prove  Theorem~\ref{thm:expanding}, under the conditions given in Section~\ref{sec:expanding_iid}, in particular {\bf (LY)}, {\bf (Dec)} and {\bf (Min)}.

For this purpose, we consider first some technical lemmas regarding short returns.  For $\omega \in \Omega$, $n \ge 1$ and $\epsilon >0$, let
\[
\cE_n^\omega(\epsilon)= \left\{ x \in [0,1] \, : \, |T_\omega^n(x) - x | \le \epsilon \right\}.
\]

\begin{lem} \label{lem:short_returns} There exists $C >0$ such that for all $\omega \in \Omega$, $n \ge 1$ and $\epsilon >0$,
\[
m(\cE_n^\omega(\epsilon)) \le C \epsilon.
\]
\end{lem}

\begin{proof}
We follow the proof of \cite[Lemma 3.4]{HNT12}, conveniently adapted to our setting of random non-Markov maps. Recall that $\cA_\omega^n$ is the partition of monotonicity associated to the map $T_\omega^n$. Consider $I \in \cA_\omega^n$. Since $\inf_I  |(T_\omega^n)'| \ge \lambda^n > 1$, there exists at most one solution $x_I^\pm \in I$ to the equation 
\begin{equation} \label{eqn:equation}
T_\omega^n(x_I^\pm) = x_I^\pm \pm \epsilon,
\end{equation}
and since there is no sign change of $(T_\omega^n)'$ on $I$, we have 
\begin{equation} \label{eqn:inclusion}
\cE_n^\omega(\epsilon) \cap I \subset [x_I^-, x_I^+].
\end{equation}
We have
\[
T_\omega^n(x_I^+) - T_\omega^n(x_I^-) = x_I^+ - x_I^- + 2 \epsilon,
\]
and by the mean value theorem, 
\[
\left| T_\omega^n(x_I^+) - T_\omega^n(x_I^-) \right| = \left| (T_\omega^n)'(c) \right| \, \left| x_I^+ - x_I^- \right|, \: \text{ for some } c \in I.
\]
Consequently,
\begin{equation} \label{eqn:estimate}
\left| x_I^+ - x_I^- \right| \le \left( \sup_I \frac{1}{|(T_\omega^n)'|} \right) \left[ \left| x_I^+ - x_I^-\right| + 2 \epsilon \right] \le \lambda^{-n} \left| x_I^+ - x_I^-\right| + 2 \epsilon \sup_I \frac{1}{|(T_\omega^n)'|}.
\end{equation}
Note that if there is no solutions to \eqref{eqn:equation}, then the estimate \eqref{eqn:estimate} is actually improved.
Rearranging \eqref{eqn:estimate} and summing over $I\in \cA_\omega^n$, we obtain thanks to \eqref{eqn:inclusion}
\[
m(\cE_n^\omega(\epsilon)) \le \sum_{I \in \cA_\omega^n} \left| x_I^+ - x_I^-\right| \le \frac{2 \epsilon}{1  - \lambda^{-n}} \sum_{I \in \cA_\omega^n} \sup_I \frac{1}{|(T_\omega^n)'|} \le C \epsilon. 
\]
The fact that 
\begin{equation} \label{eqn:distortion}
\sum_{I \in \cA_\omega^n} \sup_I \frac{1}{|(T_\omega^n)'|} \le C
\end{equation}
for a constant $C >0$ independent from $\omega$ and $n$ follows from a standard distortion argument for one-dimensional maps that can be found in the proof of part 3 of \cite[Lemma 8.5]{ANV15} (see also \cite[Lemma 7]{AR16}), where finitely many piecewise $C^2$ uniformly expanding maps with finitely many discontinuities are also considered. Since it follows from {\bf (LY)} that $\| P_\omega^n f \|_{\rm BV} \le C \|f\|_{\rm BV}$ for some uniform $C >0$, we do not have to average \eqref{eqn:distortion} over $\omega$ as in \cite{ANV15}, but instead we can simply have an estimate that holds uniformly in $\omega$.
\end{proof}

Recall that, for a measurable subset $U$, $R_U^\omega(x) \ge 1$ is the hitting time of $(\omega, x)$ to $U$ defined by \eqref{eqn:hitting_times}.

\begin{lem} \label{lem:maximal}
Let $a>0$, $\frac 2 3 < \psi < 1$ and $0 < \kappa < 3 \psi - 2$. Then there exist sequences $(\gamma_1(n))_{n \ge 1}$ and $(\gamma_2(n))_{n \ge 1}$ with $\gamma_1(n) = \cO(n^{- \kappa})$ and $\gamma_2(n) = o(1)$, and for all $\omega \in \Omega$, a sequence of measurable subsets $(A_n^\omega)_{n \ge 1}$ of $[0,1]$ with $m(A_n^\omega) \le \gamma_1(n)$ and such that for all $x_0 \notin A_n^\omega$,
\[
(\log n) \sum_{i=0}^{n-1} m\left( B_{n^{-\psi}}(x_0) \cap \left\{R_{ B_{n^{-\psi}}(x_0)}^{\sigma^i \omega} \le \lfloor a \log n \rfloor \right\}\right) \le \gamma_2(n).
\]
\end{lem}

\begin{proof}
Let
\[
E_n^\omega = \left\{ x \in [0,1] \, : \, |T_\omega^j(x) - x | \le 2 n^{- \psi} \text{ for some } 0 < j \le  \lfloor a \log n \rfloor \right\}.
\]

Since $ B_{n^{-\psi}}(x_0) \cap \left\{R_{ B_{n^{-\psi}}(x_0)}^{\sigma^i \omega} \le \lfloor a \log n \rfloor \right\} \subset B_{n^{-\psi}}(x_0) \cap E_n^{\sigma^i \omega}$, it is enough to consider
\[
(\log n) \sum_{i=0}^{n-1} m\left( B_{n^{-\psi}}(x_0) \cap E_n^{\sigma^i \omega}\right).
\]

According to Lemma~\ref{lem:short_returns}, we have
\[
m(E_n^\omega) \le \sum_{j=1}^{\lfloor a \log n \rfloor } m\left( \cE_j^\omega(2n^{-\psi})\right) \le C \frac{\log n}{n^\psi}.
\]

We introduce the maximal function
\[
M_n^\omega(x_0) = \sup_{t > 0} \frac{1}{2t} \int_{x_0 - t}^{x_0+t} \left( \sum_{i=0}^{n-1} \mathds{1}_{E_n^{\sigma^i \omega}}(z) \right) dz = \sup_{t > 0} \frac{1}{2t} \sum_{i=0}^{n-1} m\left( B_t(x_0) \cap E_n^{\sigma^i \omega}\right)
\]

By \cite[Equation (5) page 138]{Rud}, for all $\lambda > 0$, we have
\begin{equation} \label{eqn:maximal}
m(M_n^\omega > \lambda) \le \frac{C}{\lambda} \left\| \sum_{i=0}^{n-1} \mathds{1}_{E_n^{\sigma^i \omega}} \right\|_{L^1_m} \le \frac{C}{\lambda} \sum_{i=0}^{n-1} m(E_n^{\sigma^i \omega}) \le \frac{C}{\lambda} \frac{\log n}{n^{\psi - 1}}
\end{equation}

Let $\rho>0$ and $\xi >0$ to be determined later. We define
\[
F_n^\omega = \left\{ x_0 \in [0,1] \, : \, m\left( B_{n^{-\psi}}(x_0) \cap E_n^\omega \right) \ge 2n^{-\psi(1+\rho)} \right\}, 
\]
so that we have
\[
\sum_{i=0}^{n-1} m\left( B_{n^{-\psi}}(x_0) \cap E_n^{\sigma^i \omega}\right) \ge \left( \sum_{i=0}^{n-1} \mathds{1}_{F_n^{\sigma^i \omega}}(x_0)\right) 2 n^{-\psi(1 + \rho)}.
\]

By definition of the maximal function $M_n^\omega$, this implies that
\[
M_n^\omega(x_0) \ge n^{-\psi \rho} \left( \sum_{i=0}^{n-1} \mathds{1}_{F_n^{\sigma^i \omega}}(x_0) \right),
\]
from which it follows, by \eqref{eqn:maximal} with $\lambda = ( \log n) n^{\xi - \psi \rho}$,
\[
m \left(A_n^\omega \right) \le m \left( M_n^\omega > (\log n) n^{\xi - \psi \rho} \right) \le C n^{-(\xi + (1 - \rho)\psi -1)} =: \gamma_1(n),
\]
where 
\[
A_n^\omega = \left\{ \left( \sum_{i=0}^{n-1} \mathds{1}_{F_n^{\sigma^i \omega}} \right) > (\log n) n^\xi \right\}.
\]

If $x_0 \notin A_n^\omega$, then
\begin{align*}
(\log n) \sum_{i=0}^{n-1} m\left( B_{n^{-\psi}}(x_0) \cap E_n^{\sigma^i \omega}\right) & \le (\log n) \left( \sum_{i=0}^{n-1} \mathds{1}_{F_n^{\sigma^i \omega}}(x_0)\right) m(B_{n^{-\psi}}(x_0)) + 2 (\log n) n^{1-\psi(1+\rho)} \\
& \le C (\log n) \left( (\log n) n^{-(\psi - \xi)} + n^{-(\psi(1+\rho) - 1)} \right) =: \gamma_2(n).
\end{align*}

Since $\frac{2}{3} < \psi < 1$ and $0 < \kappa < 3 \psi - 2$, it is possible to choose $\rho >0$ and $\xi>0$ such that $\kappa = \xi + (1-\rho)\psi - 1$, $\psi > \xi$ and $\psi(1+\rho) >1$ \footnote{For instance, take $\xi = \psi - \delta$ and $\rho = \psi^{-1} - 1 + \delta \psi^{-1}$ with $\delta = \frac{3 \psi - 2 - \kappa}{2}$.}, which concludes the proof.
\end{proof}

\begin{lem} \label{lem:tech_short}
Suppose that $a >0$ and $\frac 3 4 < \psi < 1$. Then for $m$-a.e.~$x_0 \in [0,1]$ and $\bP$-a.e.~$\omega \in \Omega$ and , we have
\[
\lim_{n \to \infty} (\log n) \sum_{i=0}^{n-1} m\left( B_{n^{-\psi}}(x_0) \cap \left\{R_{ B_{n^{-\psi}}(x_0)}^{\sigma^i \omega} \le \lfloor a \log n \rfloor \right\}\right) = 0.
\]
\end{lem}

\begin{proof}
Let $0 < \kappa < 3 \psi - 2$ to be determined later. Consider the sets $(A_n^\omega)_{n \ge 1}$ given by Lemma~\ref{lem:maximal}, with $m(A_n^\omega) \le \gamma_1(n) = \cO(n^{- \kappa})$. Since $\kappa < 1$,  we need to consider a subsequence $(n_k)_{k \ge 1}$ such that $\sum_{k \ge 1} \gamma_1(n_k) < \infty$. For such a subsequence, by the Borel-Cantelli lemma, for $m$-a.e.~$x_0$, there exists $K = K(x_0, \omega)$ such that for all $k \ge K$, $x_0 \notin A_{n_k}^\omega$. Since $\lim_{k \to \infty} \gamma_2(n_k) = 0$, this implies 
\[
\lim_{k \to \infty} (\log n_k) \sum_{i=0}^{n_k - 1} m\left( B_{n_k^{-\psi}}(x_0) \cap \left\{R_{ B_{n_k^{-\psi}}(x_0)}^{\sigma^i \omega} \le \lfloor a \log n_k \rfloor \right\}\right) = 0.
\]
We take $n_k = \lfloor k^\zeta \rfloor$, for some $\zeta >0$ to be determined later. In order to have $\sum_{k \ge 1} \gamma_1(n_k) < \infty$, we need to require that $\kappa \zeta> 1$.
Set $U_n^\omega(x_0) =  B_{n^{-\psi}}(x_0) \cap \left\{R_{ B_{n^{-\psi}}(x_0)}^\omega \le \lfloor a \log n \rfloor \right\}$. To obtain the convergence to $0$ of the whole sequence, we need to prove that
\begin{equation} \label{eqn:subsequence_BC}
\lim_{k \to \infty} \sup_{n_k \le n < n_{k+1}} \left| (\log n) \sum_{i=0}^{n-1} m(U_n^{\sigma^i \omega}(x_0)) - (\log n_k ) \sum_{i=0}^{n_k - 1} m(U_{n_k}^{\sigma^i \omega}(x_0)) \right| = 0.
\end{equation}
For this purpose, we estimate 
\[
\left| (\log n) \sum_{i=0}^{n-1} m(U_n^{\sigma^i \omega}(x_0)) - (\log n_k ) \sum_{i=0}^{n_k - 1} m(U_{n_k}^{\sigma^i \omega}(x_0)) \right| \le {\rm (I)} + {\rm (II)} + {\rm (III)} + {\rm (IV)} + {\rm (V)}.
\]
where
\[
{\rm (I)} = \left| \log n - \log n_k\right| \sum_{i=0}^{n-1} m(U_n^{\sigma^i \omega}(x_0)), \: \:
{\rm (II)} = (\log n_k) \sum_{i=n_k}^{n-1} m(U_n^{\sigma^i \omega}(x_0)),
\]
\[
{\rm (III)} = (\log n_k) \sum_{i=0}^{n_k - 1} \left| m\left(B_{n^{-\psi}}(x_0) \cap \left\{R_{ B_{n^{-\psi}}(x_0)}^{\sigma^i \omega} \le\lfloor  a \log n \rfloor \right\}\right) - m\left(B_{n_k^{-\psi}}(x_0) \cap \left\{R_{ B_{n^{-\psi}}(x_0)}^{\sigma^i \omega} \le \lfloor a \log n \rfloor \right\}\right)\right|,
\]
\[
{\rm (IV)} = (\log n_k) \sum_{i=0}^{n_k - 1} \left|m\left(B_{n_k^{-\psi}}(x_0) \cap \left\{R_{ B_{n^{-\psi}}(x_0)}^{\sigma^i \omega} \le \lfloor a \log n \rfloor \right\}\right) - m\left(B_{n_k^{-\psi}}(x_0) \cap \left\{R_{ B_{n_k^{-\psi}}(x_0)}^{\sigma^i \omega} \le \lfloor a \log n \rfloor \right\}\right) \right|,
\]
\[
{\rm (V)} = (\log n_k) \sum_{i=0}^{n_k - 1} \left| m\left( B_{n_k^{-\psi}}(x_0) \cap \left\{R_{ B_{n_k^{-\psi}}(x_0)}^{\sigma^i \omega} \le \lfloor a \log n \rfloor \right\} \right) - m\left( B_{n_k^{-\psi}}(x_0) \cap \left\{R_{ B_{n_k^{-\psi}}(x_0)}^{\sigma^i \omega} \le \lfloor a \log n_k \rfloor \right\} \right) \right|.
\]
Before proceeding to estimate each term, we note that $|n_{k+1} - n_k| = \cO(k^{-(1- \zeta )})$, $|n_{k+1}^{- \psi} - n_k^{-\psi}| = \cO(k^{-(1 + \zeta \psi )})$, $\left| \log n_{k+1} - \log n_k \right| = \cO(k^{-1})$ and $m(U_n^\omega(x_0)) \le m(B_{n^{-\psi}}(x_0)) = \cO(k^{-\zeta \psi})$.

From these observations, it follows
\[
{\rm (I)} \le C \left| \log n_{k+1} - \log n_k \right| n_{k+1}  k^{- \zeta \psi} \le C k^{-(1-(1 - \psi) \zeta)},
\]
\[
{\rm (II)} \le C (\log n_k) |n_{k+1} - n_k| k^{-\zeta \psi} \le C (\log k) k^{-(1-(1 - \psi) \zeta)},
\]
\[
{\rm (III)} \le C (\log n_k) n_k m(B_{n_k^{-\psi}}(x_0) \setminus B_{n^{-\psi}}(x_0)) \le C (\log n_k) n_k | n_{k+1}^{-\psi} - n_k^{-\psi}| \le C (\log k) k^{-(1 - (1 - \psi) \zeta)},
\]
\begin{align*}
{\rm (IV)} & \le C (\log n_k) \sum_{i=0}^{n_k - 1} m \left(B_{n_k^{-\psi}}(x_0) \cap \left\{ R_{B_{n_k^{-\psi}}(x_0) \setminus B_{n^{-\psi}}(x_0)}^{\sigma^i \omega} \le \lfloor a \log n \rfloor \right\} \right) \\
& \le C (\log n_k) \sum_{i=0}^{n_k - 1} a (\log n) m\left(B_{n_k^{-\psi}}(x_0) \setminus B_{n^{-\psi}}(x_0)\right) \\
& \le C (\log k)^2 k^{-(1 - (1 - \psi) \zeta)}
\end{align*}
and
\begin{align*}
{\rm (V)} & \le C (\log n_k) \sum_{i=0}^{n_k -1} m\left( B_{n_k^{-\psi}}(x_0) \cap \left\{ \lfloor a \log n_k \rfloor < R_{ B_{n_k^{-\psi}}(x_0)}^{\sigma^i \omega} \le \lfloor a \log n \rfloor \right\}\right) \\
& \le C (\log n_k) \sum_{i=0}^{n_k-1 } a \left| \log n_{k+1} - \log n_k \right| m(B_{n_k^{-\psi}}(x_0)) \\
& \le C (\log k) k^{-(1 - (1 - \psi) \zeta)}.
\end{align*}

To obtain \eqref{eqn:subsequence_BC}, it is thus sufficient to choose $\kappa >0$ and $\zeta>0$ such that $\kappa < 3 \psi - 2$, $\kappa \zeta > 1$ and $(1 - \psi) \zeta < 1$, which is possible if $\psi > \frac 3 4$.
\end{proof}
%
%
%
We can now prove the functional convergence to a L\'evy stable process for i.i.d.~uniformly expanding maps.

\begin{proof}[Proof of Theorem~\ref{thm:expanding}]
  We apply Theorem \ref{thm:iid_main}. By Theorem
  \ref{thm:poisson_expanding}, we have $N_n^\omega \dto N_{(\alpha)}$ under
  the probability $\nu^\omega$ for $\bP$-a.e.~$\omega \in \Omega$. It thus
  remains to check that equation \eqref{eqn:ergodicity_iid} holds for $m$-a.e.~$x_0$ when $\alpha \in [1,2)$ to complete the proof. For this purpose,
  we will use a reverse martingale argument from \cite{NTV18} (see also
  \cite[Proposition 13]{AR16}). Because of \eqref{eqn:density_bounded}, it
  is enough to work on the probability space $([0,1], \nu^\omega)$ for
  $\bP$-a.e.~$\omega \in \Omega$. Let $\mathcal{B}$ denote the
  $\sigma$-algebra of Borel sets on $[0,1]$ and
\[
\mathcal{B}_{\omega, k} =(T_{\omega}^k)^{-1}(\mathcal{B})
\]
To simplify notation a bit let 
\[
f_{\omega,j, n} (x)=\phi_{x_0} (x) \one_{\left\{ | \phi_{x_0} | \le \epsilon b_n \right\}}(x)-\bE_{\nu^{\sigma^j \omega}}( \phi_{x_0} \one_{\left\{ | \phi_{x_0} | \le \epsilon b_n\right\}}).
\]
From \eqref{eqn:density_bounded}, it follows that $\bE_m(|f_{\omega, j, n}|) \le C \epsilon b_n$, and from the explicit definition of $\phi$, we can estimate the total variation of $f_{\omega,j,n}$ and obtain the existence of $C > 0$, independent of $\omega$, $\epsilon$, $n$ and $j$, such that
\begin{equation} \label{eqn:bounds_fj}
\|f_{\omega, j,n}\|_{\rm BV} \le C \epsilon b_n.
\end{equation}
We define 
\[
S_{\omega, k, n}:=\sum_{j=0}^{k-1} f_{\omega, j, n}\circ T_{\omega}^j
\]
and
\begin{equation}\label{H_n}
H_{\omega, k, n} \circ T_{\omega}^{n} :=\bE_{\nu^\omega} (S_{\omega, k, n}| \mathcal{B}_{\omega, k})
\end{equation}
Hence $H_{\omega, 1, n}=0$ and an explicit formula for $H_{\omega, k, n}$ is 
\[
H_{\omega, k, n}=\frac{1}{ h_{\sigma^k \omega}} \sum_{j=0}^{k-1} P_{\sigma^j \omega}^{k-j}(f_{\omega, j, n} h_{\sigma^j \omega}).
\]
From the explicit formula, the exponential decay in the BV norm of $P_{\sigma^j \omega}^{n-j}$ from {\bf (Dec)}, \eqref{eqn:density_bounded} and \eqref{eqn:bounds_fj}, we see that $\|H_{\omega, k, n}\|_{\rm BV}\le C \epsilon b_n$, where the 
constant $C$ may be taken as constant over $\omega\in \Omega$.
If we define
\[
M_{\omega, k, n}=S_{\omega, k, n}-H_{\omega, k, n}\circ T_{\omega}^{k}
\]
then the sequence $\{ M_{\omega, k, n}\}_{k \ge 1}$ is a reverse martingale difference for the  decreasing filtration $\mathcal{B}_{\omega, k}=(T_{\omega}^n)^{-1}( \mathcal{B})$
as 
\[
\bE_{\nu^\omega} (M_{\omega, k, n} |\mathcal{B}_{\omega, k})=0
\]
The martingale reverse differences are
\[
M_{\omega, k+1, n}-M_{\omega, k, n}=\psi_{\omega, k, n} \circ T_{\omega}^k
\]
where
\[
\psi_{\omega, k, n}:=f_{\omega, k, n} +H_{\omega, k, n}-H_{\omega, k+1, n}\circ T_{\sigma^{k+1} \omega}.
\]
We see from the $L^{\infty}$ bounds on $\| H_{\omega, k, n}\|_{\infty} \le C b_n \epsilon$  and the telescoping sum that 
\begin{equation} \label{eqn:telescop}
\left|\sum_{j=0}^{k-1} \psi_{\omega, j, n}\circ T_{\omega}^j -\sum_{j=0}^{k-1} f_{\omega, j, n}\circ T_{\omega}^j \right|\le C\epsilon b_n.
\end{equation}
By Doob's martingale maximal  inequality
\[
\nu^\omega \left\{ \max_{1\le k \le n}  \left| \sum_{j=0}^{k-1} \psi_{\omega, j, n}\circ T_{\omega}^j \right|\ge b_n \delta \right\} \le \frac{1}{b_n^2 \delta^2} \bE_{\nu^\omega} \left|\sum_{j=0}^{n-1} \psi_{\omega, j, n} \circ T_{\omega}^j \right|^2.
\]

Note that 
\[
\sum_{j=0}^{n-1} \bE_{\nu^\omega} \left[\psi^2_{\omega, j, n} \circ T_{\omega}^j\right]=\bE_{\nu^\omega} \left[\sum_{j=0}^{n-1} \psi_{\omega, j, n} \circ T_{\omega}^j\right]^{2}
\]
by pairwise orthogonality of martingale reverse differences.

As in \cite[Lemma 6]{HNTV17}
\[
\bE_{\nu^\omega} [(S_{\omega, n, n})^2]=\sum_{j=0}^{n-1} \bE_{\nu^\omega} [\psi^2_{\omega, j, n} \circ T_{\omega}^j]+ \bE_{\nu^\omega}[H_{\omega, 1, n}^2]- \bE_{\nu^\omega}[H_{\omega, n, n}^2\circ T_{\omega}^{n}].
\]
So we see that

\begin{equation} \label{eqn:bound_martingale}
\nu^\omega \left\{ \max_{1\le k \le n} \left| \sum_{j=0}^{k-1} \psi_{\omega, j, n}\circ T_{\omega}^j \right|\ge b_n \delta \right\} \le \frac{1}{b_n^2 \delta^2} \bE_{\nu^\omega} [(S_{\omega, n, n})^2] +2 \frac{C^2 \epsilon^2 }{\delta^2}
\end{equation}
where we have used $\| H^2_{\omega, j, n}\|_{\infty} \le C^2 b_n^2 \epsilon^2$.

Now we estimate
\begin{equation} \label{eqn:moment_Sn}
\bE_{\nu^\omega} [(S_{\omega, n, n})^2]\le \sum_{j=0}^{n-1} \bE_{\nu^\omega}[f^2_{\omega, j, n} \circ T_{\omega}^j]+2\sum_{i=0}^{n-1} \sum_{i <  j} \bE_{ \nu^\omega} [f_{\omega, j, n} \circ T_{\omega}^j \cdot f_{\omega, i, n} \circ T_{\omega}^i].
\end{equation}

Using the equivariance of the measures $\{ \nu^\omega\}_{\omega \in \Omega}$ and \eqref{eqn:density_bounded}, we have
\begin{equation} \label{eqn:moments2}
\sum_{j=0}^{n-1} \bE_{\nu^\omega}[f^2_{\omega, j, n} \circ T_{\omega}^j]  \le C n \bE_\nu(\phi_{x_0}^2 \one_{\left\{  | \phi_{x_0} | \le \epsilon b_n \right\}})  \sim C \epsilon^{2 - \alpha} b_n^2,
\end{equation}
by Proposition \ref{prop:karamata} and that
\begin{equation*}  
  \lim_{n \to \infty} n \, \nu( | \phistar | > \lambda b_n) =
  \lambda^{-\alpha} \text{\quad for $\lambda >0$,}
\end{equation*}
since $\phistar$ is regularly varying.

On the other hand, we are going to show that for $m$-a.e.~$x_0$
\begin{align} \label{eqn:crossed_moments}
\lim_{\epsilon \to 0} \limsup_{n \to \infty} \frac{1}{b_n^2}\sum_{i=0}^{n-1} \sum_{i< j} \bE_{\nu^\omega} [f_{\omega, j, n} \circ T_{\omega}^j \cdot f_{\omega, i, n} \circ T^i_{\omega}] = 0.
\end{align}

The first observation is that, due to condition {\bf (Dec)}, $$\bE_{\nu^\omega} [f_{\omega, j, n} \circ T_{\omega}^j \cdot f_{\omega, i, n} \circ T^i_{\omega}]\le C \theta^{j-i} \|f_{\omega, i, n}\|_{\rm BV} \|f_{\omega, j, n}\|_{L^1_m} \le C \epsilon^2 b_n^2 \theta^{j-i}$$ where $\theta<1$. Hence there exists $a > 0$ independently of $n$ and $\epsilon$ such that 
$$\sum_{j-i> \lfloor a \log n \rfloor}\bE_{\nu^\omega} [f_{\omega, j, n} \circ T_{\omega}^j \cdot f_{\omega, i, n} \circ T^i_{\omega}]\le C \epsilon^2 n^{-2} b_n^2$$ and it is enough to prove that for $\epsilon > 0$,
\[
\sum_{i=0}^{n-1}  \sum_{j = i+1}^{i + \lfloor a \log n \rfloor} \bE_{\nu^\omega} [f_{\omega, j, n} \circ T_{\omega}^j \cdot f_{\omega, i, n} \circ T^i_{\omega}] = o(b_n^2) = o(n^{\frac 2 \alpha}).
\]

By construction, the term $\bE_{\nu^\omega}[f_{\omega, i, n} \circ T_\omega^i \cdot f_{\omega, j, n} \circ T_\omega^j]$ is a covariance, and since $\phi$ is positive, we can bound this quantity by $\bE_{\nu^\omega}[f \circ T_\omega^i \cdot f \circ T_\omega^j] = \bE_{\nu^{\sigma^i \omega}}[f_n \cdot f_n \circ T_{\sigma^i \omega}^{j-i}]$ where $f_n = \phi_{x_0} \mathds{1}_{\left\{| \phi_{x_0}| \le \epsilon b_n\right\}}$.
Then, since the densities are uniformly bounded by \eqref{eqn:density_bounded}, we are left to estimate 
\begin{equation} \label{eqn:sum_covariance}
\sum_{i=0}^{n-1} \sum_{j= i+1}^{i + \lfloor a \log n \rfloor} \bE_m[f_n \cdot f_n \circ T_{\sigma^i \omega}^{j-i}].
\end{equation}

Let $\frac 3 4 < \psi < 1$ and $U_n = B_{n^{-\psi}}(x_0)$.  We bound \eqref{eqn:sum_covariance} by ${\rm (I)} + {\rm (II)} + {\rm (III)}$, where
\[
{\rm (I)} = \sum_{i=0}^{n-1} \sum_{j= i+1}^{i + \lfloor a \log n \rfloor} \int_{U_n \cap (T_{\sigma^i \omega}^{j-i})^{-1}(U_n)} f_n \cdot f_n \circ T_{\sigma^i \omega}^{j-i} dm,
\]
\[
{\rm (II)} = \sum_{i=0}^{n-1} \sum_{j= i+1}^{i + \lfloor a \log n \rfloor} \int_{U_n \cap (T_{\sigma^i \omega}^{j-i})^{-1}(U_n^c)} f_n \cdot f_n \circ T_{\sigma^i \omega}^{j-i} dm
\]
and
\[
{\rm (III)} = \sum_{i=0}^{n-1} \sum_{j= i+1}^{i + \lfloor a \log n \rfloor} \int_{U_n^c} f_n \cdot f_n \circ T_{\sigma^i \omega}^{j-i} dm.
\]

Since $\|f_n\|_\infty \le \epsilon b_n$, it follows that
\begin{align*}
{\rm (I)} &\le \epsilon^2 b_n^2 \sum_{i=0}^{n-1} \sum_{j=i +1}^{i + \lfloor a \log n \rfloor }m\left(U_n \cap (T_{\sigma^i \omega}^{j-i})^{-1}(U_n)\right) \\
& \le a \epsilon^2 b_n^2 (\log n) \sum_{i=0}^{n-1} m\left(U_n \cap \left\{ R_{U_n}^{\sigma^i \omega} \le a \log n \right\} \right),
\end{align*}
which by Lemma~\ref{lem:tech_short} is a $o(b_n^2)$ as $n \to \infty$ for $m$-a.e.~$x_0$.

To estimate (II) and (III), we will use H\"{o}lder's inequality. We first observe by a direct computation that
\begin{equation} \label{eqn:computation}
\int_{U_n^c} \phi_{x_0}^2 dm = \cO(n^{\psi \left(\frac{2}{\alpha} - 1 \right)}).
\end{equation}

We consider (III) first. Let $ A = U_n^c$. We have
\begin{align} \label{eqn:term_III}
\int_{U_n^c} f_n \cdot f_n \circ T_{\sigma^i \omega}^{j-i} dm \le \int_A \phi_{x_0} \cdot f_n \circ T_{\sigma^i \omega}^{j-i} dm & \le \left( \int_A \phi_{x_0}^2 dm \right)^{\frac 1 2} \left( \int f_n^2 \circ T_{\sigma^i \omega}^{j-i} dm \right)^{\frac 1 2} \\
& \le C \left( \int_A \phi_{x_0}^2 dm \right)^{\frac 1 2} \left( \int f_n^2  dm \right)^{\frac 1 2}.
\end{align}
By \eqref{eqn:computation}, $\left( \int_A \phi_{x_0}^2 dm \right)^{\frac 1 2} \le C n^{\frac{\psi}{2} \left(\frac{2}{\alpha} - 1 \right)}$ and by Proposition~\ref{prop:karamata}, $\left( \int f_n^2  dm \right)^{\frac 1 2} \le C n^{\frac{1}{\alpha} - \frac 1 2}$. Hence we may bound
\eqref{eqn:term_III} by $C n^{\left(1+\psi\right)\left(\frac{1}{\alpha} - \frac 1 2\right)}$.

To bound (II), let  $B= U_n \cap (T_{\sigma^i \omega}^{j-i})^{-1}(U_n^c)$. Then, 
\begin{align} \label{eqn:term_II}
\int_{U_n \cap (T_{\sigma^i \omega}^{j-i})^{-1}(U_n^c)} f_n \cdot f_n \circ T_{\sigma^i \omega}^{j-i} dm \le \int_B f_n \cdot \phi_{x_0} \circ T_{\sigma^i \omega}^{j-i} dm \le  \left( \int f_n^2 dm \right)^{\frac 1 2} \left( \int_B \phi_{x_0}^2 \circ T_{\sigma^i \omega}^{j-i}  dm \right)^{\frac 1 2}.
\end{align}
As before $\left( \int f_n^2  dm \right)^{\frac 1 2} \le C n^{\frac{1}{\alpha} - \frac 1 2}$ and
\begin{align*}
\left( \int_B \phi_{x_0}^2 \circ T_{\sigma^i \omega}^{j-i}  dm \right)^{\frac 1 2} & \le \left( \int \phi_{x_0}^2 \circ T_{\sigma^i \omega}^{j-i}  \mathds{1}_{(T_{\sigma^i \omega}^{j-i})^{-1}(U_n^c)} dm \right)^{\frac 1 2} \le C \left( \int_{U_n^c} \phi_{x_0}^2 dm \right)^{\frac 1 2} \le C n^{\frac{\psi}{2} \left(\frac{2}{\alpha} - 1 \right)}
\end{align*}
by \eqref{eqn:computation}, and so \eqref{eqn:term_II} is bounded by $C n^{\left(1+\psi\right)\left(\frac{1}{\alpha} - \frac 1 2\right)}$.

It follows that ${\rm (II)} + {\rm (III)} \le C (\log n) n^{1 + \left(1+\psi\right)\left(\frac{1}{\alpha} - \frac 1 2\right)} = o(n^{\frac 2 \alpha})$, since $ \psi < 1$. This proves that \eqref{eqn:sum_covariance} is a $o(b_n^2)$ and concludes the proof of \eqref{eqn:crossed_moments}.

Finally, from \eqref{eqn:moment_Sn}, \eqref{eqn:moments2} and \eqref{eqn:crossed_moments}, we obtain
\begin{align} \label{eqn:limsup}
\lim_{\epsilon \to 0} \limsup_{n \to \infty} \frac{1}{b_n^2} \bE_{\nu^\omega}[(S_{\omega, n, n})^2] = 0,
\end{align}
which gives the result by taking the limit first in $n$ and then in $\epsilon$ in \eqref{eqn:bound_martingale}.
\end{proof}

%

\subsection{Intermittent maps: proof of Theorem~\ref{thm:intermittent}}

We prove convergence to a stable law in the setting of  intermittent maps when $\alpha \in (0,1)$.

\begin{proof}[Proof of Theorem~\ref{thm:intermittent}]

We apply Proposition~\ref{prop:quenched_stable_usual}. By
Theorem~\ref{intermittent_poisson}, it remains to
prove~\eqref{eqn:seq_levy_01}, since $\alpha \in (0,1)$. We will need an
estimate for
$\bE_{\nu^{ \omega}}( | \phi_{x_0}| \mathds{1}_{\left\{  \phi_{x_0}  \le
    \epsilon b_n\right\}})$ which is independent of $\omega$. For this
purpose, we introduce the absolutely continuous probability measure
$\nu_{\rm max}$ whose density is given by
$h_{\rm max}(x) = \kappa x^{- \gamma_{\rm max}}$. Since all densities
$h_\omega$ belong to the cone $L$, we have that
$h_\omega \le \frac{a}{\kappa} h_{\rm max}$ for all $\omega$. Thus,
\[
  \frac{1}{b_n} \sum_{j=0}^{n-1} \bE_{\nu^{\sigma^j \omega}}( \phi_{x_0}
  \mathds{1}_{\left\{ | \phi_{x_0} | \le \epsilon b_n\right\}}) \le
  \frac{n}{b_n} \frac{a}{\kappa} \bE_{\nu_{\rm max}}(  \phi_{x_0}
  \one_{\left\{ | \phi_{x_0} | \le \epsilon b_n\right\}}).
\]
We can easily verify that $\phi_{x_0}$ is regularly varying of index
$\alpha$ with respect to $\nu_{\rm max}$, with scaling sequence equal to
$(b_n)_{n \ge 1}$ up to a multiplicative constant factor. Consequently, by
Proposition \ref{prop:karamata}, we have that, for some constant $c>0$,
\[
  \bE_{\nu_{\rm max}}( \phi_{x_0} \mathds{1}_{\left\{ | \phi_{x_0} | \le
      \epsilon b_n\right\}}) \sim c \epsilon^{1-\alpha} n^{\frac{1}{\alpha}
    - 1},
\]
which implies \eqref{eqn:seq_levy_01}.  \end{proof}

\section{The annealed case}\label{sec:annealed}

In this section, we consider the annealed counterparts of our results. Even
though the annealed versions do not seem to follow immediately from the
quenched version, it is easy to obtain them from our proofs in the quenched
case. We take $\phi_{x_0} (x)=d(x,x_0)^{-\frac{1}{\alpha}}$ as before we consider
the convergence on the measure space $\Omega \times [0,1]$ with respect to
$\nu_F (d\omega, dx)=\bP (d \omega) \nu^{\omega} (dx)$. We give precise
annealed results in the case of Theorems~\ref{thm:expanding} and~\ref{thm:intermittent}, where we
consider
\[
X^a_n(\omega, x)(t) := \frac{1}{b_n} \left[ \sum_{j=0}^{\lfloor nt \rfloor - 1} \phi_{x_0}(T_\omega^j x) - t c_n\right], \: t \ge 0,
\]
viewed as a random process defined on the probability space
$(\Omega \times [0,1], \nu)$.

\begin{thm}\label{thm:annealed}
  Under the same assumptions as Theorem \ref{thm:expanding}, the random
  process $X^a_n (t)$ converges in the $J_1$ topology to the L\'evy
  $\alpha$-stable process $X_{(\alpha)}(t)$ under the probability measure
  $\nu$.
\end{thm}

\begin{proof}
  We apply \cite[Theorem 1.2]{TK-dynamical} to the skew-product system
  $(\Omega \times [0,1], F, \nu)$ and the observable $\phi_{x_0}$ naturally
  extended to $\Omega \times [0,1]$. Recall that $\nu$ is given by the
  disintegration $\nu (d \omega, dx) = \bP(d \omega) \nu^\omega(dx)$.

We have to prove that
\begin{enumerate}[(a)]
\item $N_n \dto N_{(\alpha)}$,
\item if $\alpha \in [1, 2)$, for all $\delta > 0$,
\[
  \lim_{\epsilon \to 0} \limsup_{n \to \infty} \nu \left( (\omega, x) \, :
    \, \max_{1 \le k \le n} \left| \frac{1}{b_n} \sum_{j=0}^{k-1} \left[
        \phi_{x_0}(T_\omega^{j}x) \one_{\left\{ | \phi_{x_0} \circ T_\omega^j| \le
            \epsilon b_n \right\}}(x) - \bE_{\nu}( \phi_{x_0} \one_{\left\{ |
            \phi_{x_0} | \le \epsilon b_n\right\}})\right] \right| \ge \delta
  \right) =0,
\]
\end{enumerate}

where
\[
N_n(\omega, x)(B):= N_n^\omega(x)(B) = \# \left\{ j \ge 1 \, : \, \left( \frac{j}{n}, \frac{\phi_{x_0}(T_\omega^{j-1}(x))}{b_n}\right) \in B \right\}, \; n \ge 1.
\]

To prove (a), we take
$f \in C_K^+((0, \infty) \times (\bR \setminus \{ 0 \}))$ arbitrary. Then,
by Theorem~\ref{thm:poisson_expanding}, we have for $\bP$-a.e.~$\omega$
\[
\lim_{n \to \infty} \bE_{\nu^\omega}(e^{- N_n^\omega(f)}) = \bE(e^{- N(f)}).
\]
Integrating with respect to $\bP$ and using the dominated convergence theorem yields 
\[
\lim_{n \to \infty} \bE_{\nu}(e^{-N_n(f)}) = \bE(e^{-N(f)}),
\]
which proves (a). 

To prove (b), we simply have to integrate with respect to $\bP$ in the estimates in the proof of Theorem \ref{thm:expanding}, which hold uniformly in $\omega \in \Omega$, and then to take the limits as $n \to \infty$ and $\epsilon \to 0$.
\end{proof}

Similarly, we have:

\begin{thm} \label{thm:annealed_inter}
Under the same assumptions as Theorem~\ref{thm:intermittent}, $X_n^a(1) \dto X_{(\alpha)}(1)$ under the probability measure $\nu$.
\end{thm}

\begin{proof}
We can proceed as for Theorem~\ref{thm:annealed} in order to check the assumptions of~\cite[Theorem 1.3]{TK-dynamical} for the skew-product system
  $(\Omega \times [0,1], F, \nu)$ and the observable $\phi_{x_0}$.
\end{proof}

\section{Appendix}

The observation that our distributional limit theorems hold for any measures $\mu \ll \nu^{\omega}$ follows
from  Theorem 1, Corollary 1 and Corollary 3 of Zweim\"uller's work~\cite{Zweimuller}. 

Let 
\[
S_n (x)=\frac{1}{b_n} [\sum_{j=0}^{n-1} \phi\circ T_{\omega}^j (x)-a_n ].
\]
and suppose 

\[
S_n \rightarrow_{\nu_{\omega}} Y
\]
where $Y$ is a L\'evy random variable.


We consider first the setup of intermittent maps. We will show
that for any measure $\nu$ with density $h$ i.e. $d\nu=h dm$ in the cone
$L$, in particular Lebesgue measure $m$
with $h=1$,
\[
S_n \rightarrow_{\nu} Y
\]

We focus on $m$. According to \cite[Theorem 1]{Zweimuller} it is enough to show
that
\[
\int \psi (S_n) d\nu_{\omega} -\int \psi (S_n) dm \to 0.
\]
for any $\psi: \bR \rightarrow \bR$ which is bounded and uniformly Lipschitz. 

Fix such a $\psi$ and  consider
\[
\int \psi (\frac{1}{b_n} [\sum_{j=0}^{n-1} \phi\circ T_{\omega}^j (x)-a_n ]) (h_{\omega} -1)dm
\]
\[
\le \int \psi (\frac{1}{b_n} [\sum_{j=0}^{n-1} \phi\circ T_{\sigma^k \omega}^j (x)-a_n ]) P_{\omega}^k (h_{\omega} -1)dm
\]
\[
\le \|\psi\|_{\infty} \|P_{\omega}^k (h_{\omega} -1)\|_{L^{1}(m)}.
\]
Since $\|P_{\omega}^k (h_{\omega} -1)\|_{L^{1}_m}\to 0$ in case of
Example~\ref{intermittent} and maps satisfying (LY), (Dec) and (Min) the
assertion is proved. By~\cite[Corollary 3]{Zweimuller}, the proof for
continuous time distributional limits follows immediately.

\bibliographystyle{alpha} \bibliography{references_combine}

\end{document}